\documentclass{article}
\usepackage[utf8]{inputenc}

\usepackage{amsmath, amsthm, amssymb, latexsym,epsfig,amsthm,enumerate,multicol,wasysym}
\usepackage{color}
\usepackage{graphicx}
\usepackage{nccrules}
\usepackage{xfrac}
\usepackage{hyperref}
\usepackage{textcomp}
\usepackage{tikz}
\usepackage{xcolor}
\usepackage{comment}
\usepackage{lineno}

\newtheorem{lemma}{Lemma}

\newtheorem{rmk}{Remark}

\newtheorem{theorem}{Theorem}

\newcommand{\R}{\mathbb{R}}
\newcommand{\E}{\mathbb{E}}
\newcommand{\prob}{\mathbb{P}}

\newcommand{\opt}{\text{opt}}

\title{Controlling Unknown Linear Dynamics with Bounded Multiplicative Regret\thanks{This work was supported by AFOSR grants FA9550-19-1-0005 and FA9550-18-1-0069, and partially supported by NSF grant DMS-1700180 and by the Joachim Herz Stiftung.}}
\date{September 2021}
\author{J. Carruth, M. F. Eggl, C. Fefferman, C. Rowley, M. Weber}
\renewcommand\footnotemark{}

\begin{document}

\maketitle

\begin{abstract}
    We consider a simple control problem in which the underlying dynamics depend on a parameter that is unknown and must be learned. We exhibit a control strategy which is optimal to within a multiplicative constant. While most authors find strategies which are successful as the time horizon tends to infinity, our strategy achieves lowest expected cost up to a constant factor for a fixed time horizon.
\end{abstract}

\section{Introduction}
Here, as in our previous paper \cite{Fefferman:2021}, we investigate control problems in which we must make decisions with little time and little data available. Our motivating example is the success of pilots learning in real time to fly and safely land an airplane after it has been severely damaged, as documented in \cite{Brazy:2009}.

To start to address these issues, we study a simple control problem whose dynamics depend on a single unknown parameter $a\in\R$. We try to minimize an expected cost $S(\sigma,a)$, where $\sigma$ denotes our control strategy. For known $a$, classical control theory provides an optimal strategy $\sigma_{\opt}(a)$, i.e. $\sigma \mapsto S(\sigma,a)$ achieves a minimum at $\sigma = \sigma_\opt(a)$.

Now suppose we have no prior information about $a$. It is then natural to evaluate a strategy $\sigma$ by comparing its expected cost to that of $\sigma_\opt(a)$ for each putative value of $a$. Accordingly, we define the \emph{additive} and \emph{multiplicative regret} of a given strategy $\sigma$ to be the functions
\[
AR_\sigma(a) := S(\sigma,a) - S(\sigma_\opt(a), a)
\]
and
\[
MR_\sigma(a) := \frac{S(\sigma, a)}{S(\sigma_\opt(a), a)},
\]
respectively, for all $a \in \R$. The additive regret is often called simply the ``regret'', and the multiplicative regret is also called the ``competitive ratio''. (In the control theory literature, it's common to define regret by comparing our strategy to a particular class of competing strategies. Here we allow arbitrary competing strategies.)

A natural is question is whether there exists a strategy $\sigma$ for which the multiplicative regret $MR_\sigma(a)$ is even bounded for all $a$. This is not obvious, even if the dynamics are linear, because if, for instance, the parameter $a$ determines the stability of the linear system, an arbitrarily large value of $a$ could make the cost $S(\sigma, a)$ arbitrarily large. The main result of this paper is to answer this question. In particular, for the common case of linear dynamics, we exhibit a strategy $\sigma_*$ whose multiplicative regret remains bounded as the unknown $a$ varies over the whole real line. Our strategy $\sigma_*$ controls the system while learning about $a$ on the fly.

We next comment on the state of the art regarding control with learning, and contrast the regime studied here with that considered previously. Then we formulate our control problem, state our main theorem, and describe the control strategy $\sigma_*$.

Control with learning has been considered in many application areas; see for example \cite{dean2018regret}, in which Dean et al.\ apply adaptive control techniques to learn and control linear systems on the fly; the related \cite{Cohen:2019}; as well as \cite{abeille2017thompson}, in which Abeille and Lazaric use Thompson Sampling in the same setting. Online learning and control can further be expanded to other complementary problems including the study of tracking adversarial targets \cite{abbasi2014tracking} and derivative-free optimisation for the linear quadratic problem \cite{furieri2020learning}. 

As gradient descent methods are essential in offline optimization, studies have also been undertaken that aim to apply these methods without prior information; these include the derivation and implementation of adaptive online gradient descent by Bartlett, Hazan, and Rakhlin \cite{bartlett2007adaptive}, as well as the introduction of \emph{AdaGrad} (adaptive gradient algorithm) by Duchi, Hazan, and Singer \cite{duchi2011adaptive}, which has since become a staple in the machine learning community.

We note work of Abbasi-Yadkori and Szepesv\'{a}ri \cite{abbasi2011regret}, in which the authors were able to prove that under certain assumptions the expected additive regret of the adaptive controller is bounded by $\tilde{O} (\sqrt{T})$, where $T$ denotes the time horizon. Further progress was made on this problem by Chen and Hazan \cite{chen2021black}, who, assuming controllability of the system, gave the first efficient algorithm capable of attaining sublinear additive regret in a single trajectory in the setting of online nonstochastic control.

Much is known about the closely related ``multi-armed bandit'' problem; see, for instance, the classic papers \cite{robbins1952some,vermorel2005multi}, the more recent survey \cite{Bubeck:2012}, or \cite{abernethy_competing_nodate}, in which the authors introduce an efficient algorithm which achieves the optimal sublinear additive regret for the related bandit optimization problem.

As this list of references by no means does justice to the breadth of studies in the literature, we point the reader towards \cite{hazan2019introduction} for a more thorough overview of online convex optimization.

In this paper, we consider a simple toy problem that differs from the work cited above in two respects.
\begin{itemize}
    \item The time horizon $T$ is fixed, whereas the literature is mainly concerned with asymptotic behavior as $T\rightarrow \infty$.
    \item We must control a system that may be arbitrarily unstable.
\end{itemize}

\subsection{The problem}

To state our problem, we begin by recalling the classical linear quadratic regulator (``LQR'') control problem \cite{kwakernaak1972linear} in its simplest form. We consider a particle whose movement is governed by a simple linear system driven by additive noise. We hope to keep that particle close to the origin by applying a control. The position of the particle at time $t$ is $q(t) \in \R^1$, and the control we apply is denoted by $u(t)\in\R^1$.

The particle moves according to the stochastic ODE
\begin{equation}\label{eq: q ode}
    dq = (aq+u)dt + dW_t, \qquad q(0)=0,
\end{equation}
where $a$ is a known constant and $(W_t)$ denotes Brownian motion, normalized so that $\E((W_t)^2)=t$. We are free to pick our favorite $u(t)$ in \eqref{eq: q ode}, provided only that
\begin{itemize}
    \item $u(t)$ is determined by history up to time $t$ (i.e., by $(q(\tau))_{0\le\tau\le t}$) and
    \item the stochastic ODE \eqref{eq: q ode} admits a unique solution.
\end{itemize}
We define such a choice of $u(t)$ for each $t$ and each history up to time $t$ to be a \emph{strategy}, and we write $\sigma, \sigma', \tilde{\sigma}$, etc. to denote strategies (we remark that in the control theory literature strategies are often called ``policies'').

Let us fix the parameter $a$ in \eqref{eq: q ode}. Once we pick a strategy $\sigma$, we can solve the stochastic ODE \eqref{eq: q ode} from time $t=0$ to time $t=T$; we define the \emph{cost} of our strategy $\sigma$ to be the quantity\footnote{Usually, one sees coefficients multiplying the terms $dW_t$ in \eqref{eq: q ode} and $u^2$ in \eqref{eq: intro 2}. These coefficients can be easily scaled away.}
\begin{equation}\label{eq: intro 2}
    \int_0^T\left( (q(t))^2+(u(t))^2\right)dt
\end{equation}
for a known ``time horizon'' $T$. The cost is a random variable, because the noise $(W_t)$ is random. We denote the expected value of the cost by $S(\sigma,a)$.

The classic LQR problem is to pick $\sigma$ to minimize $S(\sigma,a)$ for fixed known $a$. The minimizer $\sigma_{\text{opt}}(a)$ of $S(\sigma,a)$ for fixed, given $a$ is well-known; it prescribes the control \begin{equation}\label{eq: intro 3}
    u(t) = -K(t,a)q(t),
\end{equation}
where $K(t,a)$ solves an ODE in the $t$-variable, with $a$ appearing as a parameter. The expected cost of the optimal strategy is comparable to $a$ for $a$ large positive, and to $|a|^{-1}$ for $a$ large negative. See Section \ref{sec: S_0} below.

Now we can state the problem studied here. Suppose our particle moves according to \eqref{eq: q ode}, but we have absolutely no prior information about the coefficient $a$. How should we pick our control strategy?

\subsection{A strategy with bounded multiplicative regret}

Our main result is as follows.

\begin{theorem}\label{thm: intro 1}
The control strategy $\sigma_*$, to be described below, satisfies the inequality
\[
S(\sigma_*,a)\le C S(\sigma_{\text{opt}}(a),a) \quad \text{for all } a \in \R,
\]
where $C$ is a positive constant depending only on the time horizon $T$.
\end{theorem}

Thus, although $\sigma_*$ is independent of $a$, its expected cost differs by at most a factor $C$ from that of the optimal strategy tailored to $a$; this holds for all $a \in \R$. That is, for fixed $T$, the multiplicative regret of the strategy $\sigma_*$ is bounded as $a$ varies over the whole real line.

We explain further how the problem studied here differs from those considered in \cite{Cohen:2019,dean2018regret} or \cite{abbasi2011regret}. In the context of our simple LQR problem \eqref{eq: q ode}, \eqref{eq: intro 2}, \cite{Cohen:2019,abbasi2011regret} assume that $a$ is bounded or that a stabilizing gain is provided \cite{dean2018regret}, and produce strategies whose additive regret exhibits favorable asymptotics as $T\rightarrow \infty$.

Here, we study a simple toy problem in a different regime. We regard $T$ as fixed, and exhibit a strategy whose multiplicative regret is bounded, even if $a$ varies over the whole real line. The case $a \gg 1$ is dangerous, because the dynamics are then highly unstable if we undercontrol, while overcontrol will incur a high cost.

We next describe the strategy $\sigma_*$. We will partition the time interval $[0,T]$ into a prologue and several epochs. Note that we can execute our strategy without knowing $T$.

We begin the \emph{Prologue} at time $0$. During the Prologue we set $u(t)\equiv 0$, and observe $q(t)$. Perhaps $|q(t)|<1$ for all $t\in[0,T)$, in which case the Prologue lasts until the end of the game at time $T$. On the other hand, there may come a first time $t_0 \in (0,T)$ when $|q(t_0)|=1$. In that case, we enter \emph{Epoch 0} at time $t_0$. During Epoch 0 we continue to set $u(t)\equiv 0$.

Assuming Epoch 0 occurs, we may find that $|q(t)|<2$ for all $t \in [t_0,T)$, in which case Epoch 0 lasts until the end of the game. On the other hand, assuming Epoch 0 occurs, there may come a first time $t_1 \in (t_0,T)$ when $|q(t_1)|= 2$. At time $t_1$ we then enter \emph{Epoch 1}, during which we set
\begin{equation}\label{eq: intro 4}
a_1 = 4\frac{\ln(2)}{t_1-t_0}
\end{equation}
and apply the control
\begin{equation}\label{eq: intro 5}
    u(t) = -2 a_1 q(t).
\end{equation}

Here, $\frac{1}{4}a_1$ is a guess for the unknown parameter $a$, and \eqref{eq: intro 5} is a proxy for \eqref{eq: intro 3}. We will motivate \eqref{eq: intro 4} later.

Suppose we have entered Epoch 1. It may happen that $|q(t)|<4$ for all $t \in [t_1,T)$, in which case Epoch 1 lasts until the end of the game. On the other hand, it may happen that $|q(t_2)|=4$ at some first time $t_2 \in (t_1,T)$. In that case, we enter \emph{Epoch 2}, during which we set
\begin{equation}\label{eq: intro 6}
    a_2 = 4 \frac{\ln(2)}{t_2-t_1}+2^{12} a_1,
\end{equation}
and apply the control
\begin{equation}\label{eq: intro 7}
u(t) = - 2 a_2 q(t).
\end{equation}
We continue in this way, as many times as necessary. If we find ourselves in Epoch $\nu\ge 1$ starting at time $t_\nu \in (0,T)$, then it may happen that $|q(t)|<2^{\nu+1}$ for all $t \in [t_\nu,T)$, in which case Epoch $\nu$ lasts from time $t_\nu$ to the end of the game at time $T$. On the other hand, it may happen that $|q(t_{\nu+1})|=2^{\nu+1}$ at some first time $t_{\nu+1}\in(t_\nu,T)$. In that case, we enter Epoch $\nu+1$ at time $t_{\nu+1}$. During Epoch $(\nu+1)$, we set
\begin{equation}\label{eq: intro 8}
    a_{\nu+1} = 4 \frac{\ln 2}{t_{\nu+1}-t_\nu} + 2^{12}a_\nu
\end{equation}
and apply the control
\begin{equation}\label{eq: intro 9}
    u(t) = - 2 a_{\nu+1}q(t).
\end{equation}

With probability 1 we will reach the end of the game after finitely many epochs. This completes the description of our strategy $\sigma_*$.

We next comment on the proof of Theorem \ref{thm: intro 1}. The main challenge is to deal with the case of large positive $a$, since then \eqref{eq: q ode} may be highly unstable. Let us see how our strategy $\sigma_*$ performs in that case.

During the Prologue and Epoch 0 we set $u\equiv 0$, hence \eqref{eq: q ode} reduces to
\begin{equation}\label{eq: intro 10}
    dq = aq\ dt + dW_t, \quad q(0)=0 \quad \text{ with } a \gg 1.
\end{equation}
It is then very likely that $q(t)$ will grow rapidly. In particular, we will encounter times $t_0$ and $t_1$ with $|q(t_0)|=1$ and $|q(t_1)|=2$, as in our description of the strategy $\sigma_*$.

In Epoch 0--the time interval $[t_0, t_1]$--the stochastic ODE \eqref{eq: intro 10} tells us that $q(t) \approx q(t_0)\cdot \exp(a\cdot[t-t_0])$ with high probability. Since $|q(t_0)|=1$ and $|q(t_1)|=2$, it follows that $a \approx \frac{\ln 2}{t_1 - t_0}$ with high probability.

So, at the end of Epoch 0, we have an excellent guess for the unknown $a$. That's why we define $a_1$ by equation \eqref{eq: intro 4}; the factor of 4 in that equation is inserted as a margin of safety. As we enter Epoch 1, we very likely have $a_1 \approx 4a$, hence \eqref{eq: intro 5} yields $u \approx -8 a q$, and \eqref{eq: q ode} becomes
\begin{equation}\label{eq: intro 11}
    dq \approx - 7a q\ dt + dW_t, \text{ with } a\gg 1.
\end{equation}
Equation \eqref{eq: intro 11} leads to extremely stable behavior. In particular, since $|q(t_1)|=2$, it is highly likely that soon after time $t_1$ the system enters a regime in which $|q|\ll 1$ until the end of the game at time $T$. So we will probably encounter the Prologue and Epochs 0 and 1, but no further epochs. During the Prologue and Epoch 0 we gather enough information to guess $a$ with high confidence, and in Epoch 1 we use that guess to control the system and reverse the exponential growth of $|q|$. That's what happens with high probability.

We may be very unlucky; with small probability, it may happen that $a_1 \ll a$. In that case, \eqref{eq: intro 4} and \eqref{eq: intro 5} will lead us to undercontrol the system, and $|q|$ will continue to grow exponentially. If that disaster occurs, then at some first time $t_2\in(t_1,T)$ we will have $|q(t_2)|=4$. At time $t_2$ we then enter Epoch 2. We increase our previous guess for the unknown $a$ based on $a_1$ to a new guess based on $a_2$ given by \eqref{eq: intro 6}. In place of our previous rule \eqref{eq: intro 5}, we now define our control $u(t)$ by \eqref{eq: intro 7}. With high probability we have $a_2 > a$, which allows us to reverse the exponential growth of $|q|$ starting in Epoch 2. Very likely, then, Epoch 2 lasts until the end of the game at time 1.

However, we may be extraordinarily unlucky and end up with $a_2 \ll a$. In that exceedingly rare event, the exponential increase of $|q|$ continues unabated until we enter Epoch 3.

We continue in this way until we finally reach an Epoch in which $a_\nu > a$ and the game ends before we reach Epoch $\nu+1$. This happens with probability 1. Already the probability that we even reach Epoch 2 is exponentially small for $a \gg 1$.

So our strategy $\sigma_*$ succeeds in controlling the system when $a\gg 1$. In particular, our guesses \eqref{eq: intro 4}, \eqref{eq: intro 6}, \eqref{eq: intro 8} for $a$ and the consequent control formulas \eqref{eq: intro 5}, \eqref{eq: intro 7}, \eqref{eq: intro 9} are reasonable.

Now let's see how $\sigma_*$ performs when $a$ is large and negative. During the Prologue, we take $u=0$, so our stochastic ODE \eqref{eq: q ode} is simply
\[
dq = - |a| q\ dt + dW_t, \quad q(0)=0 \quad \text{ with } |a| \gg 1.
\]

Very likely, we never encounter $|q|=1$, and therefore the Prologue lasts until the end of the game. Thus, $\sigma_*$ will very likely tell us to set $u(t) \equiv 0$ for all $t$. The expected cost of that strategy (``do nothing'') is comparable to the cost of the optimal strategy $\sigma_{\text{opt}}(a)$ when $a$ is large negative.

Finally, suppose $a$ is neither large positive nor large negative; say $|a|\le 1$. Then $S(\sigma_{\text{opt}}(a),a)$ and $S(\sigma_*,a)$ are bounded above and below by constants depending only on the time horizon $T$, which immediately implies the conclusion of Theorem \ref{thm: intro 1}.

This concludes our introductory explanation of the proof of Theorem \ref{thm: intro 1}. Detailed proofs are given in Sections \ref{sec: prelim}-\ref{sec: a neg} below; we warn the reader that the rigorous proofs have to deal with low-probability disasters omitted from this introduction. Section \ref{sec: S_0} reviews the classical LQR problems, in particular deriving the asymptotic behavior of $S(\sigma_{\text{opt}}(a),a)$ for $|a| \gg 1$. 

We would like to obtain a strengthened form of Theorem \ref{thm: intro 1} in which the constant $C$ is as small as possible up to an arbitrarily small error $\varepsilon >0$, with the strategy $\sigma_*$ depending on $\varepsilon$. We believe that a slight variant of our present $\sigma_*$ will be one ingredient in that $\varepsilon$-dependent strategy.

We refer the reader to our previous paper \cite{Fefferman:2021} for additional problems involving control with learning on the fly.

We are grateful to Amir Ali Ahmadi and Elad Hazan for helpful conversations.

\section{Main Result and Outline of Proof}\label{sec:Main_res}

In this section we discuss the proof of Theorem \ref{thm: intro 1}. We adopt the following notation for the remainder of the paper: we write $S_\opt(a):= S(\sigma_\opt(a),a)$ and $S_*(a) := S(\sigma_*(a),a)$. In order to prove Theorem \ref{thm: intro 1}, it is necessary to understand the asymptotic behavior of $S_\opt$. This is the content of the following lemma.
\begin{lemma}\label{lem: S_0}
There exists a constant $C>0$, depending only on the time horizon $T$, such that the following hold:
\begin{enumerate}[(i)]
    \item {For all $|a| \le 1$, $S_\opt(a) \ge C$.} \label{lem: S_0 1}
    \item {For all $a \ge 1$, $ S_\opt(a) \ge C a $.} \label{lem: S_0 2}
    \item {For all $a \le - 1$, $S_\opt(a) \ge \frac{C}{|a|}$.} \label{lem: S_0 3}
\end{enumerate}
\end{lemma}
\noindent We prove Lemma \ref{lem: S_0} in Section \ref{sec: S_0}. The following lemma, combined with Lemma \ref{lem: S_0}, proves Theorem \ref{thm: intro 1}.
\begin{lemma}\label{lem: main 2}There exists a constant $C>0$, depending only on the time horizon $T$, such that the following holds:
\begin{enumerate}[(i)]
    \item{For any $a \ge 1$, $S_*(a) \le Ca$.}\label{lem: main 2 i}
    \item{For any $|a| \le 1$, $S_*(a) \le C$.}\label{lem: main 2 ii}
    \item {For any $a \le -1$, $S_*(a) \le \frac{C}{|a|}$.}\label{lem: main 2 iii}
\end{enumerate}
\end{lemma}
\noindent The proof of Lemma \ref{lem: main 2} is given in Sections \ref{sec: a pos}-\ref{sec: a neg}.

\subsection{Conventions on constants}

Throughout this paper we let $C, C', C'', \dots$ denote positive constants depending only on $T$. If we wish to specify that a positive constant depending only on $T$ is smaller than 1 we use $c, c', \dots$. These constants are not fixed throughout the paper and will change from one line to another.

Given quantities $X, Y \ge 0$, we will write either $X\lesssim Y$ or $X = O(Y)$ if there exists a constant $C>0$ depending only on $T$ such that $X \le CY$. We write $X \approx Y$ if $X\lesssim Y$ and $Y \lesssim X$.

\section{Preliminaries}\label{sec: prelim}
We begin this section by establishing some notation. For $\nu \ge 0$, we define $q_\nu :=2^\nu$. We also define $q_{-1}:=0$. For $\nu \ge -1$ we let $t_\nu$ denote the first time $t\in[0,T]$ for which $|q(t)| = q_\nu$ if such a time exists. If no such time exists we set $t_\nu = T$. Note that we always have $t_{-1} = 0$. For $\nu \ge 0$, we will refer to the time interval $[t_\nu, t_{\nu+1})$ as Epoch $\nu$. If $t_\nu < T$ then we say that \emph{Epoch $\nu$ occurs}. We let $E_\nu$ denote the event that Epoch $\nu$ occurs. We refer to the time interval $[0, t_0)$ as the Prologue; note that the Prologue always occurs.

We now recall our strategy. Let us denote our control in the Prologue by $u_{-1}(t)$ and our control in Epoch $\nu$ by $u_\nu(t)$, i.e. when Epoch $\nu$ occurs we set $u(t) = u_\nu(t)$ for $t \in [t_\nu, t_{\nu+1})$. In the Prologue and Epoch 0 we exercise no control, i.e. we set $u_{-1}, u_0 \equiv 0$. When Epoch $\nu$ occurs, we control during Epoch $\nu$ with $u_\nu(t) = -2 a_\nu q(t)$, where $a_\nu$ will be defined shortly. We define $a_{-1} = a_0 = 0$. We let $C_0$ and $C_1$ denote positive constants which will be chosen later. Then when Epoch $\nu$ occurs for $\nu \ge 1$ we define
\begin{equation}\label{eq: anu def}
    a_\nu = C_0 \frac{\log(2)}{t_\nu - t_{\nu-1}} + C_1 a_{\nu-1}.
\end{equation}
Note that when Epoch $\nu$ occurs for $\nu \ge 1$ we are guaranteed to have $a_\nu > 0$. In Lemma \ref{lem: epnu pc prob} we fix $C_0 = 4$ and $C_1 = 2^{\ell_\#}$, where $\ell_\#$ is a universal constant which is chosen to be sufficienlty large in the later sections (we will see that it suffices to take $\ell_\# = 12$, and so we can take $C_1 = 2^{12}$ as in \eqref{eq: intro 8}).

Finally, for a given value of $a$ in \eqref{eq: q ode} we will write $S_\nu(a)$ to denote the cost incurred in Epoch $\nu$ using strategy $\sigma_*$ and $S_{-1}(a)$ to denote the cost incurred in the Prologue using strategy $\sigma_*$. Specifically, we set
\[
S_\nu(a) := \int_{t_\nu}^{t_{\nu+1}} (q^2(t) + u_\nu^2(t))dt.
\]
Note that $S_\nu(a)$ is a random variable depending on the noise $(W_t)$ in \eqref{eq: q ode}.

We now state and prove some preliminary lemmas.

\begin{lemma}\label{lem: q dist}
Fix $b \ne 0$ and $\nu \ge -1$. Define
\[
X_t^{(b)}:= e^{-bt}\tilde{q}(t) - q_\nu,
\]
where $\tilde{q}$ is a random process satisfying $\tilde{q}(0)=q_\nu$ and governed by
\begin{equation}\label{eq: b law}
d\tilde{q} = (b \tilde{q})dt + d\widetilde{W}_t,
\end{equation}
where $(\widetilde{W}_t)$ is standard Brownian motion. Then the following hold:
\begin{enumerate}
    \item For fixed $t$, $X_t^{(b)}$ is a normal random variable with mean $0$ and variance $\frac{1-e^{-2bt}}{2b}$.
    \item Define
    \[
    M_t^{(b)} := \sup_{0\le s \le t} X_s^{(b)}.
    \]
Then for any $\eta >0$, 
\[
\prob(M_t^{(b)}>\eta) = 2 \prob(X_t^{(b)}>\eta).
\]
\end{enumerate}
\end{lemma}
\begin{proof}
Using an integrating factor of the form $e^{-bt}\tilde{q}$ gives
\[
d(e^{-bt}\tilde{q}) = e^{-bt}d\tilde{q}- be^{-bt}\tilde{q} dt = e^{-bt}d\widetilde{W} \; .
\]
For fixed $t$, $e^{-bt}d\widetilde{W}$ is normally distributed with mean 0 and variance $e^{-2bt} dt$. Integrating from 0 to t gives
\begin{equation}\label{eq: 3}
e^{-bt}\tilde{q}( t) - q_\nu = \int_{0}^{t} e^{-bs}d\widetilde{W_s} \; .
\end{equation}
The RHS of \eqref{eq: 3} is normally distributed with mean 0 and (by the It\^{o} isometry) variance $\frac{1- e^{-2bt}}{2b}$. This proves the first claim.

Note that if we replace $X_t^{(b)}$ by Brownian motion then the second claim is simply the reflection principle for Brownian motion (see, for example, \cite{feller}). Therefore we refer to this claim as the reflection principle for $X_t^{(b)}$. Note, however, that for fixed $b$, $X_{t}^{(b)}$ has a.s. continuous paths. Moreover, $X_t^{(b)}$ satisfies the strong Markov property and $(X_{t+t'}^{(b)} - X_{t}^{(b)}) \sim e^{-bt} X_{t'}^{(b)}$ is symmetric. Examining the proof of the reflection principle for Brownian motion in \cite{feller} shows that these conditions are sufficient to prove the reflection principle for $X_t^{(b)}$.
\end{proof}

\begin{rmk}\label{rmk: sym}
We first note that if, in Lemma \ref{lem: q dist}, we assume instead that $\tilde{q}(0) = -q_\nu$ then, due to the symmetry of \eqref{eq: b law}, we have that $(e^{-bt} \tilde{q}( t) + q_\nu )$ is distributed the same as $X_t^{(b)}$ given $\tilde{q}(0) = q_\nu$.

We also note that, when we apply Lemma \ref{lem: q dist}, we will be taking $\tilde{q}(t) = q(t_\nu + t)$ in Epoch $\nu$, conditioned on an event that specifies $a_\nu$ and is independent of the noise in Epoch $\nu$. In this case, we will be taking $b = a - 2a_\nu$.
\end{rmk}

\begin{rmk}\label{rmk: e approx}
We record here a few simple observations which we will use repeatedly throughout this paper.
First, we note that there exists a constant $c \in (0,1)$ such that 
\begin{enumerate}
    \item[(a)] {When $|bt| \le 1/10$, then
\[
c\cdot(bt) \le 1 - e^{-bt} \le bt.
\]}
\item[(b)]{When $bt \ge 1/10$, then
\[
c \le 1 - e^{-bt} < 1.
\]}
\end{enumerate}

Second, for any positive integer $N$ there exists a constant $C_N$ such that
\[
\exp(-x) \le C_N\cdot x^{-N}
\]
for all $x>0$.

Finally, we note that there exist constants $C>0$ and $c\in(0,1)$ such that if $Y$ is a normal random variable with mean 0 and standard deviation $\sigma$ then for any $x >0$ we have
\[
\prob( |Y| > x \sigma) \le C \exp(-c x^2).
\]
\end{rmk}

\begin{lemma}\label{lem: error bound}
Fix $\delta > 0$ and $\nu \ge 0$, and suppose $\tilde{q}$ is governed by \eqref{eq: b law} with $b \ne 0$. Suppose that $|\tilde{q}(t_\nu)| = 2^\nu$. Then
\begin{enumerate}
    \item {If $b> 0$, then for each $t>0$ the probability that $|\tilde{q}(t_\nu+t)| \notin [2^\nu e^{bt}(1-\delta), 2^\nu e^{bt} (1+\delta)]$ is at most $C e^{-c \delta^2 2^{2\nu} b}$.}
    \item{For each $t \in \left[0, \frac{1}{10|b|}\right]$, the probability that $|\tilde{q}(t_\nu+t)| \notin  [2^\nu e^{bt}(1-\delta), 2^\nu e^{bt}(1+\delta)]$ is at most $C e^{-c\delta^2 2^{2\nu} /t}$.}
\end{enumerate}
\end{lemma}
\begin{proof}
We write $I:=[2^\nu e^{bt}(1-\delta), 2^\nu e^{bt} (1+\delta)]$. First note that
\begin{equation}\label{eq: err bound eq6}
\begin{split}
\prob ( |\tilde{q}(t_\nu+t)| \notin I )
&\le \prob( |\tilde{q}(t_\nu+t)| \notin I | \tilde{q}(t_\nu) = 2^\nu)\\
&+ \prob( |\tilde{q}(t_\nu+t)| \notin I | \tilde{q}(t_\nu) = -2^\nu).
\end{split}
\end{equation}
Now, since the event $|\tilde{q}(t_\nu+t)| \notin I$ implies the events $\tilde{q}(t_\nu+t) \notin I$ and $-\tilde{q}(t_\nu+t) \notin I$, we continue from \eqref{eq: err bound eq6} to get
\begin{equation}\label{eq: err bound extra}
\begin{split}
    \prob ( |\tilde{q}(t_\nu+t)| \notin I)
    &\le \prob( \tilde{q}(t_\nu+t) \notin I | \tilde{q}(t_\nu) = 2^\nu)\\
&+ \prob( -\tilde{q}(t_\nu+t) \notin I | \tilde{q}(t_\nu) = -2^\nu).
\end{split}
\end{equation}
Note that when $\tilde{q}(t_\nu) = 2^\nu$,  $\tilde{q}(t_\nu + t) \notin I$ means that $e^{-bt}\tilde{q}(t_\nu + t) - 2^\nu \notin [-2^\nu \delta, 2^\nu \delta]$. Therefore
\begin{equation}\label{eq: err bound 1 new}
\prob(\tilde{q}(t_\nu + t) \notin I | \tilde{q}(t_\nu) = 2^\nu) \le \prob( X_t^{(b)} \notin [-2^\nu \delta, 2^\nu \delta] ),
\end{equation}
where $X_t^{(b)}$ is as defined in Lemma \ref{lem: q dist}. Similarly, when $\tilde{q}(t_\nu) = -2^\nu$, $-\tilde{q}(t_\nu+t) \notin I$ means that
\begin{equation}\label{eq: err bound eq7}
-e^{-bt}\tilde{q}(t_\nu + t) - 2^\nu \notin [-2^\nu\delta, 2^\nu \delta].
\end{equation}
By Remark \ref{rmk: sym}, 
\begin{equation}\label{eq: err bound eq8}
\begin{split}
\prob(\tilde{q}(t_\nu + t) + 2^\nu \notin [-2^\nu\delta, 2^\nu \delta] | \tilde{q}(t_\nu) = -2^\nu)
= \prob (X_t^{(b)}\notin [-2^\nu\delta,2^\nu\delta] ).
\end{split}
\end{equation}
Combining \eqref{eq: err bound extra}--\eqref{eq: err bound eq8} gives
\begin{equation}\label{eq: err bound eq5}
\begin{split}
\prob(|\tilde{q}(t_\nu+t)| \notin I )\lesssim\prob( X_t^{(b)} \notin [-2^\nu \delta, 2^\nu \delta]).
\end{split}
\end{equation}
By Remark \ref{rmk: e approx} and Lemma \ref{lem: q dist}, when $|tb|\le 1/10$ the standard deviation of $X_{t}^{(b)}$ is 
\[
\sigma = \left( \frac{1 - e^{-2bt}}{2b}\right)^{1/2} \approx t^{1/2},
\]
and when $tb \ge 1/10$ the standard deviation of $X_{t}^{(b)}$ is
\[
\sigma = \left( \frac{1 - e^{-2bt}}{2b}\right)^{1/2} \approx \frac{1}{b^{1/2}}.
\]
Note that when $X_t^{(b)} \notin [-2^\nu \delta, 2^\nu \delta]$ holds, the normal random variable $X_{t}^{(b)}$ is at least $N := \frac{2^\nu \delta}{\sigma}$ standard deviations from its mean. By Remark \ref{rmk: e approx}, this implies that the probability of $X_t^{(b)} \notin [-2^\nu \delta, 2^\nu \delta]$  holding for a given $t$ is at most $C \exp(-c N^2)$. Therefore
\begin{equation}\label{eq: err bound eq3}
\prob(X_{t}^{(b)} \notin [-2^\nu \delta, 2^\nu \delta]) \lesssim 
\begin{cases}
    \exp\left(-\frac{c2^{2\nu}\delta^2}{t}\right) & \text{ when } t|b| < 1/10 \\
    \exp\left(-c 2^{2\nu} \delta^2 b\right) & \text{ when } tb > 1/10
\end{cases}.
\end{equation}
Note that $t|b| < 1/10$ implies that
\begin{equation}\label{eq: err bound eq4}
\exp\left(-\frac{c2^{2\nu}\delta^2}{t}\right) \lesssim \exp(-c 2^{2\nu} \delta^2 b);
\end{equation}
combining \eqref{eq: err bound eq5}, \eqref{eq: err bound eq3}, and \eqref{eq: err bound eq4} proves the lemma.

\end{proof}

\begin{rmk}\label{rmk: occur}
Throughout this note we will analyze probabilities which are conditioned on events of the form $a_\nu \in I$ for some interval $I$. We clarify here that this event will always mean ``Epoch $\nu$ occurs and $a_\nu \in I$''.
\end{rmk}

Before stating the next lemma, we remind the reader of the constants $C_0$ and $C_1$ appearing in the definition of $a_\nu$ (see Equation \eqref{eq: anu def}). Recall that $C_0 = 4$ and $C_1 = 2^{\ell_\#}$, where $\ell_\#$ is a universal constant which is chosen to be sufficiently large in the later sections.

\begin{lemma}\label{lem: beta gamma}
Let $\nu \ge 1$ and $a \in \R$. Let $X \subset [0, \infty)$, $\gamma: X \rightarrow (0, \infty)$, $\beta: X \rightarrow (0, \infty)$. Suppose
\[
|a - 2 x| \le \gamma(x)
\]
and
\[
\frac{\gamma(x)}{\beta(x)} < \frac{1}{10 C_0 \log(2)}
\]
for all $x\in X$. Let
\[
\beta^* = \inf_{x\in X} \beta(x).
\]
Then
\[
\prob\left( (a_\nu - C_1 a_{\nu-1}) \ge \beta(a_{\nu-1}) | a_{\nu-1}\in X\right) \lesssim \exp\left( - c 2^{2\nu} \beta^*\right).
\]
\end{lemma}
\begin{proof}
Fix $\tilde{a} \in X$. Using the definition of $a_\nu$, we have
\begin{equation}\label{eq: beta gamma add 1}
    \begin{split}
\prob\left( a_\nu - C_1 \tilde{a} \right. &\left. \ge  \beta(\tilde{a}) | (a_{\nu-1} = \tilde{a})\right) \\
&= \prob\left(\left. \frac{C_0 \log(2)}{t_\nu - t_{\nu-1}} \ge \beta(\tilde{a}) \right.| (a_{\nu-1}=\tilde{a})\right)\\
&= \prob(\exists t \in [0, t^*(\tilde{a})] : |q(t_{\nu-1} + t)| > 2^\nu | (a_{\nu-1}=\tilde{a}))
\end{split}
\end{equation}
where $t^*:X \rightarrow \R$ is defined by
\begin{equation}\label{eq: tstar}
t^*(x) := \frac{C_0\log(2)}{\beta(x)}.
\end{equation}
Using our hypotheses on $\gamma$ and $\beta$ gives
\begin{equation}\label{eq: bt bound}
|(a - 2 \tilde{a})t^*(\tilde{a})| < \frac{1}{10}.
\end{equation}
Now observe that
\begin{equation}\label{eq: later sign 1}
\begin{split}
    \prob(&\exists t \in [0, t^*(\tilde{a})]:|q(t_{\nu-1}+t)| > 2^\nu | (a_{\nu-1} = \tilde{a}))\\
    &\le \prob(\exists t\in [0,t^*(\tilde{a})]:q(t_{\nu-1}+t)>2^\nu|(a_{\nu-1}=\tilde{a})\cap(q(t_{\nu-1})=2^{\nu-1}))\\
    &+ \prob(\exists t\in [0,t^*(\tilde{a})]:q(t_{\nu-1}+t)< -2^\nu|(a_{\nu-1}=\tilde{a})\cap(q(t_{\nu-1})=2^{\nu-1}))\\
    &+ \prob(\exists t\in [0,t^*(\tilde{a})]:q(t_{\nu-1}+t)>2^\nu|(a_{\nu-1}=\tilde{a})\cap(q(t_{\nu-1})=-2^{\nu-1}))\\
    &+\prob(\exists t\in [0,t^*(\tilde{a})]:q(t_{\nu-1}+t)<-2^\nu|(a_{\nu-1}=\tilde{a})\cap(q(t_{\nu-1})=-2^{\nu-1})).\\
\end{split}
\end{equation}
If $t \in [0, t^*(\tilde{a})]$ and $q(t_{\nu-1}+t)>2^\nu$, then by Remark \ref{rmk: e approx} and \eqref{eq: bt bound} we have
\begin{equation}\label{eq: later sign 1.5}
    q(t_{\nu-1} + t)e^{(a-2\tilde{a})t} - 2^{\nu-1} > 2^{\nu-1}\left( 2\cdot\frac{9}{10}-1 \right)> c 2^{\nu-1}
\end{equation}
and
\begin{equation}\label{eq: later sign 2}
    q(t_{\nu-1}+t)e^{(a-2\tilde{a})t} + 2^{\nu-1} > 2^{\nu-1} > c 2^{\nu-1}
\end{equation}
for all $t\in[0,t^*(\tilde{a})]$. Similarly, if $t\in[0,t^*(\tilde{a})]$ and $q(t_{\nu-1} + t) < -2^\nu$ then
\begin{equation}\label{eq: later sign 3}
    q(t_{\nu-1}+t)e^{(a-2\tilde{a})t} - 2^{\nu-1} < -  2^{\nu-1} < -c 2^{\nu-1}
\end{equation}
and
\begin{equation}\label{eq: later sign 4}
    q(t_{\nu-1} + t)e^{(a-2\tilde{a})t}+2^{\nu-1}<-2^{\nu-1}\left(2 \cdot \frac{9}{10}-1\right) < -c2^{\nu-1}.
\end{equation}
Combining \eqref{eq: later sign 1} through \eqref{eq: later sign 4} with the definition of $X_t^{(b)}$ and Remark \ref{rmk: sym} gives
\begin{equation}\label{eq: beta gamma add 2}
    \begin{split}
        \prob(&\exists t\in [0,t^*(a_{\nu-1})]:|q(t_{\nu-1}+t)|>2^\nu | a_{\nu-1}=\tilde{a})\\
        &\lesssim \prob(\exists t\in[0,t^*(\tilde{a})]: X_t^{(a-2\tilde{a})}>c 2^{\nu-1}).
    \end{split}
\end{equation}
The reflection principle for $X_t^{(b)}$ (the second part of Lemma \ref{lem: q dist}) gives
\begin{equation}
\begin{split}
\prob(\exists t\in[0,t^*(\tilde{a})]: X_t^{(a-2\tilde{a})}>c 2^{\nu-1}) \approx \prob(X_{t^*(\tilde{a})}^{(a-2\tilde{a})}>c2^{\nu-1}).
\end{split}
\end{equation}
By \eqref{eq: err bound eq3}, \eqref{eq: tstar}, and \eqref{eq: bt bound} we have
\begin{equation}\label{eq: beta gamma add 3}
\prob( X_{t^*(\tilde{a})}^{(a-2\tilde{a})} > \tilde{c} 2^{\nu-1} )\lesssim\exp\left(-\frac{c2^{2\nu}}{t^*(\tilde{a})}\right)\lesssim \exp\left( - c' 2^{2\nu} \beta^* \right).
\end{equation}
Combining \eqref{eq: beta gamma add 1} with \eqref{eq: beta gamma add 2} - \eqref{eq: beta gamma add 3} finishes the proof of the lemma.
\end{proof}

Before stating the next lemma we remind the reader that we write $E_\nu$ to denote the event that we reach Epoch $\nu$.

\begin{lemma}\label{lem: epoch nu large neg a}
Let $m, M \in \R$ satisfy $0 < m< M$. Let $\nu \ge 1$ be an integer. Let $X\subset [0,\infty)$ such that if $\tilde{a} \in X$ then $m < 2\tilde{a} - a < M$. Then
\[
\prob(E_\nu | a_{\nu-1}\in X) \lesssim M \exp(-c 2^{2\nu}m).
\]
\end{lemma}
\begin{proof}
Fix $\tilde{a} \in X$ and $\tilde{t} \in (0,T)$. Define $\tilde{b}:= a- 2\tilde{a}$, $\Delta t := \frac{1}{10 M}$, $N:= \lfloor \frac{T-\tilde{t}}{\Delta t}\rfloor$, and $I_j := [\alpha_j, \beta_j]$ for $j=0,1,\dots,N$; here $\alpha_j := j \Delta t$ for $j=0,1,\dots, N$, $\beta_j = (j+1)\Delta t$ for $j = 0, 1, \dots, N-1$, and $\beta_N =(T-\tilde{t})$. We remark that $\beta_N \le (N+1)\Delta t$. Note that $\tilde{b} \in (-M,-m)$. We have
\begin{equation}\label{eq: exist_t 2}
\begin{split}
\prob( E_\nu |& (a_{\nu-1}=\tilde{a})\cap(t_{\nu-1}=\tilde{t}))\\
&=
    \prob( \exists t\in(0,T-\tilde{t}) : |q(\tilde{t}+t)|>2^\nu | (a_{\nu-1}=\tilde{a})\cap(t_{\nu-1}=\tilde{t}))\\
    &\le \sum_{j=0}^{N} \prob( \exists t\in I_j : |q(\tilde{t}+t)|>2^\nu | (a_{\nu-1}=\tilde{a})\cap(t_{\nu-1}=\tilde{t})).
    \end{split}
\end{equation}
We claim that
\begin{equation}\label{eq: exist_t 3}
    \prob( \exists t\in I_j : |q(\tilde{t}+t)|>2^\nu | (a_{\nu-1}=\tilde{a})\cap(t_{\nu-1}=\tilde{t})) \lesssim \exp\left(-c 2^{2\nu}m\right).
\end{equation}
Combining \eqref{eq: exist_t 2} and \eqref{eq: exist_t 3}, and using the fact that $N \lesssim M$ gives
\[
\prob\left( E_\nu |  (a_{\nu-1}=\tilde{a})\cap(t_{\nu-1}=\tilde{t})\right) \lesssim M \exp\left(-c 2^{2\nu} m\right).
\]
This proves the lemma. Thus it just remains to establish \eqref{eq: exist_t 3}. We begin by noting that
\begin{equation}\label{eq: exist_t 4}
    \begin{split}
        \prob( \exists t&\in I_j : |q(\tilde{t}+t)|>2^\nu | (a_{\nu-1}=\tilde{a})\cap(t_{\nu-1}=\tilde{t})) \\
        &\lesssim \prob( \exists t\in I_j : q(\tilde{t}+t)>2^\nu | (a_{\nu-1}=\tilde{a})\cap(t_{\nu-1}=\tilde{t}) \cap (q(t_{\nu-1}) = 2^{\nu-1})) \\
        & + \prob( \exists t\in I_j : q(\tilde{t}+t)<-2^\nu | (a_{\nu-1}=\tilde{a})\cap(t_{\nu-1}=\tilde{t}) \cap (q(t_{\nu-1}) = 2^{\nu-1}))\\
        &+ \prob( \exists t\in I_j : q(\tilde{t}+t)>2^\nu |(a_{\nu-1}=\tilde{a})\cap(t_{\nu-1}=\tilde{t}) \cap (q(t_{\nu-1}) = -2^{\nu-1})) \\
        & + \prob( \exists t\in I_j : q(\tilde{t}+t)<-2^\nu | (a_{\nu-1}=\tilde{a})\cap(t_{\nu-1}=\tilde{t}) \cap (q(t_{\nu-1}) = -2^{\nu-1}))
    \end{split}
\end{equation}
Since $\tilde{b} < 0$, we have $e^{-\tilde{b}t} \ge e^{-\tilde{b}j\Delta t}$ for $t \in I_j$. Thus the existence of $t'\in I_j$ for which $q(\tilde{t} + t') > 2^\nu$ implies that
\begin{equation}\label{eq: exist_t 5}
q(\tilde{t} + t') e^{-\tilde{b}t'} - 2^{\nu-1} \gtrsim 2^\nu e^{-\tilde{b}t'} \ge  2^\nu e^{-\tilde{b}j\Delta t}.
\end{equation}
Similarly, if there exists $t'\in I_j$ for which $q(\tilde{t}+t')<-2^\nu$ then we have
\begin{equation}\label{eq: exist_t 6}
q(\tilde{t} +t')e^{-\tilde{b}t'} - 2^{\nu-1} < -2^\nu e^{-\tilde{b}t'} - 2^{\nu-1} < -2^\nu e^{-\tilde{b}t'} \le -2^\nu e^{-\tilde{b}j\Delta t}.
\end{equation}
This shows that the first two terms on the right hand side of \eqref{eq: exist_t 4} are each bounded by
\begin{equation}\label{eq: exist_t 7}
\prob(\exists t \in I_j : X_t^{(\tilde{b})} > c 2^\nu e^{-\tilde{b}j\Delta t}  ),
\end{equation}
where $X_t^{(\tilde{b})}$ is as in Lemma \ref{lem: q dist}.

By Remark \ref{rmk: sym}, $q(t_{\nu-1}+t)e^{-\tilde{b}t} + 2^{\nu-1}$ given $q(t_{\nu-1}) = -2^{\nu-1}$ has the same distribution as $X_t^{(\tilde{b})}$. Using this observation, and arguing as in \eqref{eq: exist_t 5} and \eqref{eq: exist_t 6}, shows that the third and fourth terms on the right hand side of \eqref{eq: exist_t 4} are also each bounded by \eqref{eq: exist_t 7}. Therefore
\begin{equation}\label{eq: exist_t 8}
\begin{split}
    \prob( \exists t\in I_j &: |q(\tilde{t}+t)|>2^\nu | (a_{\nu-1}=\tilde{a})\cap(t_{\nu-1}=\tilde{t}))\\ &\lesssim
    \prob(\exists t \in I_j : X_t^{(\tilde{b})} > c 2^\nu e^{-\tilde{b}j\Delta t}  ).
    \end{split}
\end{equation}
The right hand side of \eqref{eq: exist_t 8} is bounded by
\begin{equation}\label{eq: exist_t 9}
    \prob(\exists t \in [0,\beta_j] : X_t^{(\tilde{b})} > c 2^\nu e^{-\tilde{b}j\Delta t} ),
\end{equation}
which, by the reflection principle for $X_t^{(\tilde{b})}$ (the second part of Lemma \ref{lem: q dist}), is equal to
\begin{equation}\label{eq: exist_t 10}
    2 \prob(X_{\beta_j}^{(\tilde{b})} > c 2^\nu e^{-\tilde{b}j\Delta t}).
\end{equation}
By Lemma \ref{lem: q dist}, the standard deviation of $X_{\beta_j}^{(\tilde{b})}$ is
\[
\left( \frac{1- e^{-2\tilde{b}\beta_j}}{2\tilde{b}}\right)^{1/2}.
\]
Note that Remark \ref{rmk: e approx} and the fact that $|\tilde{b}\Delta t| < 1/10$ tell us that $e^{\tilde{b}\Delta t} \ge (1 + \tilde{b} \Delta t) > c$. Also, since $\tilde{b}<0$, and since $(j+1)\Delta t \ge \beta_j$ for all $0 \le j \le N$, we have $|e^{2\tilde{b}(j+1)\Delta t} - e^{2\tilde{b}((j+1)\Delta t - \beta_j)}|\le 1$ for all $0 \le j \le N$. Therefore
\[
\frac{c2^\nu e^{-\tilde{b}j\Delta t} |\tilde{b}|^{1/2}}{|1-e^{-2\tilde{b}\beta_j}|^{1/2}} = \frac{c2^\nu |\tilde{b}|^{1/2} e^{\tilde{b}\Delta t}}{|e^{2\tilde{b}(j+1)\Delta t} - e^{2\tilde{b}((j+1)\Delta t - \beta_j)}|^{1/2}} \ge c'2^\nu |\tilde{b}|^{1/2} \ge c'' 2^\nu m^{1/2}.
\]
Thus when $X_{\beta_j}^{(\tilde{b})} > c 2^\nu e^{-\tilde{b}j\Delta t}$, the normal random variable $X_{\beta_j}^{(\tilde{b})}$ is at least $c 2^\nu m^{1/2}$ standard deviations from its mean. Therefore, by Remark \ref{rmk: e approx},
\begin{align*}
\prob( X_{\beta_j}^{(\tilde{b})} > c 2^\nu e^{-\tilde{b}j\Delta t}) &\le C \exp(-c 2^{2\nu}m).
\end{align*}
This completes the proof of \eqref{eq: exist_t 3}, thereby completing the proof of the lemma.
\end{proof}

\begin{lemma}\label{lem:gen score f}
Let $M\geq0$ and $m>0$ be real numbers. Let $\nu \ge -1$ be an integer. Let $X$ be an event which is determined by $a_\nu$ and $a_{\nu-1}$ and which implies that $\frac{a+m}{2} \le a_\nu \le M$. Then
\[
\E[S_\nu(a) | X] \lesssim \frac{(1+M^2)(1+q_\nu^2)}{m}.
\]
\end{lemma}

\begin{proof}[Proof]
Fix $\tilde{a}$ with $\frac{a+m}{2} \le \tilde{a} \le M$. Since the event $X$ is independent of the noise in Epoch $\nu$ we can apply Lemma \ref{lem: q dist} to get that
\begin{equation}\label{eq: q2 f}
    \E\left[q^2(t_\nu+t)|(a_\nu = \tilde{a}) \cap X \right] = q^{2}_{\nu}e^{2(a-2\tilde{a})t} + \frac{e^{2(a-2\tilde{a})t}-1}{2(a-2\tilde{a})}
\end{equation}
(since $q(t_\nu + t)$ is governed by $dq = (a - 2\tilde{a})qdt + dW_t$ and $|q(t_\nu)| = q_\nu$). Let $\chi_I$ denote the indicator function of an interval $I \subset \R$. Note that
\begin{equation}\label{eq: gen score f}
\begin{split}
    \E[S_\nu(a) | (a_\nu = \tilde{a}) \cap X] &\lesssim \left(1 + M^2\right) \E\left[\left.\int_{0}^{t_{\nu+1}-t_\nu} q^2(t_\nu+t)dt\right| (a_\nu = \tilde{a}) \cap X\right] \\
    &= \left(1+M^2\right) \E\left[\left.\int_{0}^{T} q^2(t_\nu+t)\chi_{[0,t_{\nu+1}-t_\nu]}(t)dt\right| (a_\nu = \tilde{a}) \cap X\right] \\
    & = \left(1+M^2\right) \int_0^T \E\left[ q^2(t_\nu+t) \chi_{[0,t_{\nu+1}-t_\nu]}(t) | (a_\nu = \tilde{a}) \cap X\right] dt\\
    &\le \left(1+M^2\right) \int_0^T q^{2}_{\nu}e^{2(a-2\tilde{a})t} + \frac{e^{2(a-2\tilde{a})t}-1}{2(a-2\tilde{a})}dt,
    \end{split}
\end{equation}
where the last line follows from \eqref{eq: q2 f}. Next, note that
\begin{equation}\label{eq: q term 1 f}
\int_0^{T} q^{2}_{\nu}e^{2(a-2\tilde{a})t}dt = \frac{q^{2}_{\nu}\left(e^{(2(a-2\tilde{a})T} - 1\right)}{2(a-2\tilde{a})}.
\end{equation}
Since $2\tilde{a} - a > m>0$, we have
\begin{equation}\label{eq: q term 2 f}
\frac{q^{2}_{\nu}\left(e^{(2(a-2\tilde{a})T} - 1\right)}{2(a-2\tilde{a})} \lesssim \frac{q^{2}_{\nu}}{m}
\end{equation}
and
\begin{equation}\label{eq: q term 3 f}
    \frac{e^{2(a-2\tilde{a})t}-1}{2(a-2\tilde{a})} \lesssim \frac{1}{m}
\end{equation}
for all $t \in [0, T]$. Combining \eqref{eq: gen score f}--\eqref{eq: q term 3 f} gives
\begin{align*}
    \E[S_\nu(a) | (a_\nu = \tilde{a}) \cap X] &\lesssim  \left(1+M^2\right)\left(\frac{q^{2}_{\nu}+1}{m}\right)\\
\end{align*}
for any $\frac{a+m}{2}<\tilde{a}\le M$. Since $X$ implies that $\frac{a+m}{2} < \tilde{a} \le M$, we have that
\[
\E[S_\nu(a) | X] \le \frac{(1+M^2)(q_\nu^2+1)}{m}.
\]
\end{proof}

\begin{lemma}\label{lem: prob-small-uc}
Let $\nu \ge 1$ and $a > 0$. Fix $\tilde{a} \le \frac{1}{4}a$, $\tilde{t} \in [0,T)$, and $\delta \in (0,1)$. Define
\[
t' := \frac{\log(2) - \log(1+\delta)}{a-2\tilde{a}}
\]
and
\[
t'' := \frac{\log(2) - \log(1-\delta)}{a - 2\tilde{a}} \; .
\]
Then
\begin{enumerate}
    \item If $\tilde{t}+t' \le T$, then
    \[
    \prob(t' < t_\nu - \tilde{t}< t'' |(a_{\nu-1}=\tilde{a})\cap(t_{\nu-1} = \tilde{t})) \ge 1 - C \exp(-2^{2\nu} a \delta^2 c).
    \]
    \item If $\tilde{t} + t' \ge T$, then
    \[
    \prob(E_\nu | (a_{\nu-1}=\tilde{a})\cap(t_{\nu-1} = \tilde{t})) \le C \exp(-2^{2\nu}a \delta^2 c).
    \]
\end{enumerate}
\end{lemma}

\begin{proof} First, observe that when $\tilde{t}+t' \le T$, we have \begin{equation}\label{eq: prob-small new 1}
\begin{split}
\prob( t_\nu - &\tilde{t} < t'\left| (a_{\nu-1} =\tilde{a})\cap(t_{\nu-1} = \tilde{t})\right) \\ &\le \prob \left( \exists \tau \in [0, t']:|q(\tilde{t} + \tau)|>2^\nu \left| (a_{\nu-1}=\tilde{a})\cap(t_{\nu-1}=\tilde{t}) \right. \right).
\end{split}
\end{equation}
Define $b:= (a-2\tilde{a})$. Note that $\tilde{a}\le \frac{1}{4}a$ implies that $b \ge \frac{1}{2}a$. In particular, $b>0$ and thus we have $e^{b\tau}(1+\delta) \le e^{b t'}(1+\delta) = 2$ for any $\tau \in [0, t']$. Therefore, 
\begin{equation}\label{eq: small-uc-eq1}
\begin{split}
    \prob&\left(\exists \tau  \in [0, t']: |q(\tilde{t}+\tau)|>2^\nu \left| (a_{\nu-1}=\tilde{a})\cap(t_{\nu-1}=\tilde{t}) \right.\right) \\ &\le
    \prob\left(\exists \tau \in [0, t']: |q(\tilde{t}+\tau)|>2^{\nu-1}e^{b\tau}(1+\delta)\left| (a_{\nu-1}=\tilde{a})\cap(t_{\nu-1}=\tilde{t})\right.\right)
    \\ &\le \prob\left( \exists \tau \in [0,t']: |q(\tilde{t}+\tau)|\notin I \left| (a_{\nu-1}=\tilde{a})\cap(t_{\nu-1}=\tilde{t})\right.\right),
    \end{split}
\end{equation}
where we write $I$ to denote the interval $[2^{\nu-1}e^{b\tau}(1-\delta), 2^{\nu-1}e^{b\tau}(1+\delta)]$.
Arguing as in equations \eqref{eq: err bound eq6} through \eqref{eq: err bound eq5}, we have
\begin{equation}\label{eq: small-uc-eq2}
\begin{split}
\prob&\left(\exists \tau  \in [0, t']: |q(\tilde{t}+\tau)| \notin I  \left| (a_{\nu-1}=\tilde{a})\cap(t_{\nu-1}=\tilde{t})\right.\right) \\
& \lesssim \prob\left(\exists \tau \in [0, t']: |X_\tau^{(b)}|>2^{\nu-1}\delta  \right).
\end{split}
\end{equation}
By the reflection principle for $X_t^{(b)}$ (the second part of Lemma \ref{lem: q dist}), 
\begin{equation}\label{eq: small-uc-eq3}
\begin{split}
\prob&\left(\exists \tau \in [0, t']: |X_\tau^{(b)}|>2^{\nu-1}\delta  \right)\\ &\approx \prob\left( |X_{t'}^{(b)} | >2^{\nu-1}\delta \right).
\end{split}
\end{equation}
Combining \eqref{eq: err bound eq3} with \eqref{eq: small-uc-eq1}--\eqref{eq: small-uc-eq3}, and using that $b \ge \frac{1}{2}a$, gives
\begin{equation*}
\begin{split}
\prob\left(\exists \tau \right.&\left. \in [0, t']: |q(\tilde{t}+\tau)|>2^\nu \left| (a_{\nu-1}=\tilde{a})\cap(t_{\nu-1}=\tilde{t})\right.\right)\\& \lesssim  \exp\left(-c {2}^{2\nu} \delta^2 b\right) \lesssim \exp(-c'2^{2\nu}\delta^2 a).
\end{split}
\end{equation*}
Combining this with \eqref{eq: prob-small new 1} implies that
\begin{equation}\label{eq: small-uc-eq4}
\prob( t_\nu - \tilde{t} < t'\left| (a_{\nu-1} =\tilde{a})\cap(t_{\nu-1} = \tilde{t})\right)  \lesssim \exp(-c'{2}^{2\nu}\delta^2 a)
\end{equation}
when $\tilde{t} + t'\le T$. Note that
\begin{equation*}
    \begin{split}
\prob(E_\nu | (a_{\nu-1}=\tilde{a})&\cap(t_{\nu-1}=\tilde{t})) \\&= \prob(\exists \tau \in [0, T-\tilde{t}): q(\tilde{t}+\tau)|>2^\nu|(a_{\nu-1}=\tilde{a})\cap(t_{\nu-1}=\tilde{t})).
\end{split}
\end{equation*}
Arguing as in the proof of \eqref{eq: small-uc-eq4} shows that we have
\[
\prob( E_\nu \left| (a_{\nu-1} =\tilde{a})\cap(t_{\nu-1} = \tilde{t})\right)  \lesssim \exp(-c'{2}^{2\nu}\delta^2 a)
\]
when $\tilde{t}+t' \ge T$; this proves the second assertion of the lemma.

Note that if $\tilde{t} + t'' > T$, then
\[
\prob(t_\nu > \tilde{t} + t'') = 0.
\]
This, combined with \eqref{eq: small-uc-eq4}, proves the first assertion of the lemma when $\tilde{t} + t'' > T$. Now observe that when $\tilde{t} +t'' \le T$ we have
\begin{equation}\label{eq: prob-small new 2}
\begin{split}
\prob( (t_\nu - \tilde{t}) > t''&\left| (a_{\nu-1} =\tilde{a})\cap(t_{\nu-1} = \tilde{t})\right) \\ &= \prob \left( \forall \tau \in [0, t''), |q(\tilde{t} + \tau)|<2^\nu \left| (a_{\nu-1} =\tilde{a})\cap(t_{\nu-1} = \tilde{t}) \right. \right)\\
&\le \prob \left( |q(\tilde{t} + t'']|<2^\nu \left| (a_{\nu-1} =\tilde{a})\cap(t_{\nu-1} = \tilde{t}) \right. \right).
\end{split}
\end{equation}
Since $2^\nu = 2^{\nu-1}e^{bt''}(1-\delta)$,
\begin{equation*}
\begin{split}
\prob&\left( |q(\tilde{t}+t'')|<2^\nu \left| (a_{\nu-1} =\tilde{a})\cap(t_{\nu-1} = \tilde{t})\right.\right)  \\&= \prob\left(|q(\tilde{t}+t'')|<2^{\nu-1}e^{bt''}(1-\delta)\left|(a_{\nu-1} =\tilde{a})\cap(t_{\nu-1} = \tilde{t})\right.\right).
\end{split}
\end{equation*}
By Lemma \ref{lem: error bound}, and again using that $b \ge \frac{1}{2}a$,
\begin{equation*}
\begin{split}
    \prob&\left(|q(\tilde{t}+t'')|<2^{\nu-1}e^{bt''}(1-\delta)\left|(a_{\nu-1} =\tilde{a})\cap(t_{\nu-1} = \tilde{t})\right.\right) \\&\lesssim \exp(-c\delta^2 2^{2\nu}b) \lesssim \exp(-c'\delta^2 {2}^{2\nu}a).
    \end{split}
\end{equation*}
Therefore,
\begin{equation}\label{eq: small-uc-eq5}
    \prob( t_\nu - \tilde{t} > t''\left| (a_{\nu-1} =\tilde{a})\cap(t_{\nu-1} = \tilde{t})\right) \lesssim \exp\left(-c \delta^2 2^{2\nu} a\right).
\end{equation}
Combining \eqref{eq: small-uc-eq4} and \eqref{eq: small-uc-eq5} finishes the proof of the first assertion of the lemma.
\end{proof}

%{\color{blue}Maybe streamline "Preliminaries"-subsections in sections~\ref{sec: a med}-~\ref{a:large-neg}? Many lemmas are similar and could likely be generalized and added to this section.}

\section{When $a$ is large and positive}\label{sec: a pos}
The goal of this section is to prove Lemma \ref{lem: main 2}.\ref{lem: main 2 i}. In particular, we will show that
\begin{equation}\label{eq: a big main}
S_*(a) \lesssim a.
\end{equation}
when $a\ge 1$. We define:
\begin{align*}
    A_o &:= \{ b\in \R : 2^{\ell_\#}a \le b\} \\
    A_p &:= \left\{ b \in \R :  a \le b < 2^{\ell_\#}a\right\}\\
    A_u &:= \left\{ b \in \R : 0 \le b < a\right\}
\end{align*}
where $\ell^\#$ is as in Section \ref{sec: prelim}. We have
\begin{equation}\label{eq: S_nu decomp}
\begin{split}
\E[S_\nu(a)] &= \E[S_\nu(a) | a_\nu \in A_o] \cdot \prob(a_\nu \in A_o) \\
&+ \E[S_\nu(a) | a_\nu \in A_p] \cdot \prob(a_\nu \in A_p) \\
&+ \E[S_\nu(a) | a_\nu \in A_u] \cdot \prob(a_\nu\in A_u).
\end{split}
\end{equation}
We now state a lemma which will allow us to control \eqref{eq: S_nu decomp}.

\begin{lemma}\label{lem: S_nu a pos}
If $a \ge 1$ and $\nu \ge 1$, then
\begin{enumerate}[(i)]
\item $\E[S_\nu(a) | a_\nu \in A_u] \cdot \prob(a_\nu \in A_u) \lesssim \frac{1}{2^{2\nu}}$,\label{lem: S_nu a pos i}
\item $\E[S_\nu(a) | a_\nu \in A_p] \cdot \prob(a_\nu \in A_p) \lesssim \frac{a}{2^{2\nu}}$, and\label{lem: S_nu a pos ii}
\item $\E[S_\nu(a) | a_\nu \in A_o] \cdot \prob(a_\nu \in A_o) \lesssim \frac{1}{2^{2\nu}}$.\label{lem: S_nu a pos iii}
\end{enumerate}
\end{lemma}

\noindent Note that we have
\[
S_*(a) = \sum_{\nu=-1}^\infty \E[S_\nu(a)].
\]
Combining this with \eqref{eq: S_nu decomp} and the lemma above gives
\[
\sum_{\nu \ge 1}\E[S_\nu(a)] \lesssim \sum_{\nu=1}^\infty \frac{a}{2^{2\nu}} \lesssim a.
\]
We note that $\E[S_{-1}(a)] = O(1)$ and $\E[S_0(a)] = O(1)$. So once we prove Lemma \ref{lem: S_nu a pos}, we'll have proved \eqref{eq: a big main}. The proof of Lemma \ref{lem: S_nu a pos} is contained in Sections \ref{sec: pc uc}--\ref{sec: oc}. It relies on the following lemma, which we devote the remainder of this section to proving.

Recall that $E_\nu$ denotes the event that Epoch $\nu$ occurs. We write $E_\nu^c$ to denote the complement of this event, i.e. the event that Epoch $\nu$ does not occur.

\begin{lemma}\label{lem: epnu pc prob}
If $\nu \ge 1$, $a > 0$, and $0<\tilde{a} \le \frac{1}{4}a$, then
\[
\prob\left((a_\nu \in A_p) \cup E_\nu^c \left| a_{\nu-1} = \tilde{a}\right. \right) \ge 1 - C \exp(- c 2^{2\nu}  a).
\]
\end{lemma}
\begin{proof}
Fix $\tilde{t} \in (0,T)$ and write $b = a - 2\tilde{a}$. Define
\begin{equation}\label{eq: epnu pc prob 1}
t' = \frac{\log(2) - \log\left(1+\frac{1}{2}\right)}{b} = \frac{\log\left(\frac{4}{3}\right)}{b}
\end{equation}
and
\begin{equation}\label{eq: epnu pc prob 2}
t'' = \frac{\log(2) - \log\left(1-\frac{1}{2}\right)}{b} = \frac{2\log(2)}{b}.
\end{equation}
Then, since $\tilde{a} \le \frac{1}{4} a$, Lemma \ref{lem: prob-small-uc} tells us that when $\tilde{t} + t' \ge T$ we have
\begin{equation}\label{eq: epnu pc prob 3}
\prob( E_\nu | (a_{\nu-1}=\tilde{a}) \cap (t_{\nu-1} = \tilde{t})) \lesssim \exp(-c2^{2\nu}a).
\end{equation}
When $\tilde{t} + t' \le T$, then Lemma \ref{lem: prob-small-uc} tells us that
\begin{equation}\label{eq: epnu pc prob 3.5}
\prob( t' < t_\nu - \tilde{t} < t''|(a_{\nu-1}=\tilde{a})\cap(t_{\nu-1}=\tilde{t})) \ge 1 - C\exp(-2^{2\nu}ac).
\end{equation}
Recall the definition of $a_\nu$:
\[
a_\nu = C_0 \frac{\log(2)}{t_\nu - t_{\nu-1}} + C_1 a_{\nu-1}.
\]
Combining this with \eqref{eq: epnu pc prob 3.5} tells us that after conditioning on the events $(a_{\nu-1}=\tilde{a})$ and $(t_{\nu-1}=\tilde{t})$, we have that
\begin{equation}\label{eq: epnu pc prob 4}
C_0\frac{\log(2)}{t''} \le a_\nu \le C_0\frac{\log(2)}{t'} + \frac{C_1}{4}a
\end{equation}
holds with probability $\ge 1 - C \exp(-2^{2\nu} a c)$. For reasons that will become clear throughout the proof of Lemma \ref{lem: main 2}, we'd like $C_1 = 2^{\ell_\#}$. We would like to guarantee that \eqref{eq: epnu pc prob 4} implies $a_\nu \in A_p$, therefore we need to choose $\ell_\#$ and $C_0$ so that
\begin{equation}\label{eq: epnu pc prob 5}
a \le C_0 \frac{\log(2)}{t''}
\end{equation}
and
\begin{equation}\label{eq: epnu pc prob 6}
C_0 \frac{\log(2)}{t'} < 2^{\ell_\#} a - \frac{C_1}{4}a = 2^{\ell_\#}a\frac{3}{4}.
\end{equation}
Since $b \ge \frac{1}{2}a$ (this follows from the hypothesis that $\tilde{a}\le \frac{1}{4}a$), we have
\[
C_0 \frac{\log(2)}{t''} =  C_0 \frac{b}{2} \ge C_0 \frac{1}{4}a.
\]
Therefore setting $C_0 = 4$ ensures that \eqref{eq: epnu pc prob 5} holds. Note that since $b < a$,
\[
C_0 \frac{4\log(2)}{3t'} =  \frac{16\log(2)b}{3\log\left(\frac{4}{3}\right)} < 13 a.
\]
 Thus, provided $\ell_\#\ge 4$, we have
\[
C_0 \frac{4\log(2)}{3t'} < 2^{\ell_\#} a
\]
and \eqref{eq: epnu pc prob 6} holds. This proves the lemma in the remaining case $\tilde{t}+t'\le T$.
\end{proof}

\subsection{Proof of Lemma \ref{lem: S_nu a pos}.\ref{lem: S_nu a pos i} and \ref{lem: S_nu a pos}.\ref{lem: S_nu a pos ii}}\label{sec: pc uc}

\begin{proof}[Proof of Lemma \ref{lem: S_nu a pos}.\ref{lem: S_nu a pos i}]
Recall the definition of $a_\nu$:
\[
a_\nu = C_0 \frac{\log(\rho)}{t_\nu - t_{\nu-1}} + C_1 a_{\nu-1},
\]
where $C_0 = 4$ and $C_1 = 2^{\ell_\#}$. Notice that $a_\nu \ge C_1 a_{\nu-1}$ and we can choose $\ell_\#$ to ensure $C_1 = 2^{\ell_\#} > 4$. Therefore $a_\nu \in A_u$ implies that
\[
a_{\nu-1} \le \frac{a}{C_1} < \frac{a}{4},
\]
and so we have
\begin{align*}
\prob(a_\nu \in A_u) &= \prob\left( (a_\nu \in A_u) \cap \left( a_{\nu-1} \le \frac{a}{4}\right)\right) \\
&\le \prob\left( a_\nu \in A_u \left| a_{\nu-1} \le \frac{a}{4}\right.\right).
\end{align*}
Applying Lemma \ref{lem: epnu pc prob} gives
\begin{align*}
\prob \left( a_\nu \in A_u \left| a_{\nu-1} \le \frac{1}{4} a \right. \right) &\le 1 - \prob \left( (a_\nu \in A_p)\cup E_\nu^c \left| a_{\nu-1} \le \frac{1}{4} a \right. \right) \\
& \lesssim \exp(-c2^{2\nu}a).
\end{align*}
Combining these gives that
\begin{equation}\label{eq: uc later eq1}
\prob(a_\nu \in A_u) \lesssim \exp(-c2^{2\nu}a)
\end{equation}
for $\nu \ge 1$. Recall that we write $\chi_I$ to denote the indicator function of an interval $I\subset \R$. Note that since $a_\nu \in A_u$ implies that $a_\nu \le a$, we have that
\begin{align*}
    \E [S_\nu(a) | a_\nu\in A_u] &= \E \left[\left.(1+(2a_\nu)^2)\int_{t_\nu}^{t_{\nu+1}}q^2(t)dt\right\vert a_\nu \in A_u \right]\\
    &\lesssim a^2 \int_0^T \E[ q^2(t) \cdot \chi_{[t_\nu, t_{\nu+1}]}(t) | a_\nu \in A_u] dt \\
    & \lesssim a^2 2^{2(\nu+1)},
\end{align*}
where the last inequality uses the fact that $|q(t)|\le  2^{\nu+1}$ for $t_\nu \le t \le t_{\nu+1}$. Combining this with \eqref{eq: uc later eq1} and applying Remark \ref{rmk: e approx} finishes the proof.
\end{proof}

\begin{proof}[Proof of Lemma \ref{lem: S_nu a pos}.\ref{lem: S_nu a pos ii} ]
By applying Lemma \ref{lem:gen score f}, with $M = 2^{\ell_\#}a$ and $m = a$ we obtain
\begin{equation}\label{eq: pc later eq1}
\E[S_\nu(a) | a_\nu \in A_p ] \lesssim a2^{2 \nu}.
\end{equation}
Using the trivial bound $\prob(a_1 \in A_p) \le 1$, we see that
\begin{equation}\label{eq: pc later eq1.1}
\E[S_1(a) | a_1 \in A_p] \cdot \prob( a_1 \in A_p) \lesssim a.
\end{equation}
We now claim that
\begin{equation}\label{eq: pc later eq2}
\prob(a_\nu \in A_p) = O\left(\frac{1}{a^22^{4\nu}}\right)
\end{equation}
for $\nu > 1$. As in the proof of Lemma \ref{lem: S_nu a pos}.\ref{lem: S_nu a pos i}, we observe that $a_\nu \ge C_1 a_{\nu-1}$. Note that if $a_\nu \in A_p$, then $a_\nu < 2^{\ell_\#} a$; our choice of $C_1 = 2^{\ell_\#}$ therefore guarantees that $a_{\nu-1} < a$ when $a_\nu \in A_p$. Therefore
\begin{equation}
    \prob\left(a_\nu \in A_p\right) = \prob\left( \left(a_\nu \in A_p\right) \cap \left(a_{\nu-1}\in A_u\right)\right) \le \prob\left( a_{\nu-1} \in A_u \right).
\end{equation}
By \eqref{eq: uc later eq1} and Remark \ref{rmk: e approx}, we have that
\begin{equation}\label{eq: ep nu pc prob eq3}
\prob(a_{\nu-1}\in A_u) \lesssim \exp\left(-c 2^{2\nu} a\right)= O\left(\frac{1}{a^22^{4\nu}}\right)
\end{equation}
for $\nu > 1$. This proves \eqref{eq: pc later eq2}. Combining \eqref{eq: pc later eq1} and \eqref{eq: pc later eq2}, we get
\[
\E[S_\nu(a) | a_\nu \in A_p] \cdot \prob(a_\nu \in A_p) \lesssim  \frac{1}{a 2^{2\nu}} \le \frac{a}{2^{2\nu}}
\]
when $\nu > 1$. Combining this with \eqref{eq: pc later eq1.1} proves the lemma.
\end{proof}

\subsection{Proof of Lemma \ref{lem: S_nu a pos}.\ref{lem: S_nu a pos iii}}\label{sec: oc}
We begin this section by introducing some notation. We define, for integers $\ell \ge \ell_\#$ and $\nu > 1$,
\begin{align*}
    A_{o,\nu}^\ell &:= \{ b\in \R: 2^\ell a_{\nu-1} \le b < 2^{\ell+1}a_{\nu-1}\}\\
    A_o^\ell &:= \{ b \in \R : 2^\ell a \le b < 2^{\ell+1} a\}
\end{align*}
As in the proof of Lemma \ref{lem: S_nu a pos}.\ref{lem: S_nu a pos i}, note that $a_\nu \ge 2^{\ell_\#}a_{\nu-1}$. Therefore, when $\nu > 1$, we have
\begin{equation}\label{eq: ep nu oc score decomp}
\begin{split}
    \E[S_\nu(a) &| a_\nu \in A_o] \cdot \prob(a_\nu \in A_o) \\
     =& \sum_{\substack{\ell_1 \ge \ell_\# \\ \ell_2 \ge \ell_\#}} \E\left[S_\nu(a) | (a_\nu \in A_{o,\nu}^{\ell_2})\cap(a_{\nu-1}\in A_o^{\ell_1})\right]\cdot \prob\left( (a_\nu \in A_{o, \nu}^{\ell_2})\cap(a_{\nu-1} \in A_o^{\ell_1})\right)\\
    & + \sum_{\ell \ge \ell_\#} \E\left[S_\nu(a) | (a_\nu \in A_{o,\nu}^\ell) \cap (a_{\nu-1} \in A_p)\right]\cdot \prob\left((a_\nu \in A_{o,\nu}^\ell)\cap (a_{\nu-1} \in A_p)\right) \\
    & + \sum_{\ell \ge \ell_\#} \E\left[ S_\nu(a) | (a_\nu \in A_o^\ell) \cap ( a_{\nu-1} \in A_u)\right] \cdot \prob\left((a_\nu \in A_o^\ell)\cap(a_{\nu-1}\in A_u)\right)
\end{split}
\end{equation}
When $\nu = 1$, we have instead
\begin{equation}\label{eq: ep 1 oc score decomp}
\begin{split}
\E[ &S_1(a) | a_1 \in A_o] \cdot \prob (a_1 \in A_o)\\
&= \sum_{\ell\ge \ell_\#} \E[S_1(a) | (a_1 \in A_o^\ell) \cap (a_{0} \in A_u)] \cdot \prob((a_\nu \in A_o^\ell) \cap (a_0\in A_u)).
\end{split}
\end{equation}
This is because $a_0 = 0$, and so $a_0 \in A_u$. We write $a_0 \in A_u$ in \eqref{eq: ep 1 oc score decomp} rather than $a_0 = 0$ to be consistent with \eqref{eq: ep nu oc score decomp}.
\begin{proof}[Proof of Lemma \ref{lem: S_nu a pos}.\ref{lem: S_nu a pos iii}]
Let $\nu \ge 1$. We will apply Lemma \ref{lem:gen score f} three times. First, when $\ell_1 \ge \ell_\#$ and $\ell_2\ge \ell_\#$ we take $M = 2^{\ell_2+\ell_1+2}a$ and $m = (2^{\ell_1+\ell_2}-1)a$ to get
\begin{equation}\label{eq: later oc eq1}
\E[S_\nu(a) | (a_\nu \in A_{o,\nu}^{\ell_2})\cap(a_{\nu-1}\in A_o^{\ell_1})]\lesssim 2^{\ell_1+\ell_2}a 2^{2\nu}.
\end{equation}
Next, when $\ell \ge \ell_\#$ we take $M = 2^{\ell+\ell_\#+1}a$ and $m = (2^\ell - 1)a$ to get
\begin{equation}\label{eq: later oc eq2}
\E[S_\nu(a) | (a_\nu \in A_{o,\nu}^\ell)\cap (a_{\nu-1} \in A_p)] \lesssim 2^\ell a 2^{2\nu}.
\end{equation}
Last, when $\ell \ge \ell_\#$, we take $M = 2^{\ell  + 1}a$ and $m = (2^\ell-1)a$ to get 
\begin{equation}\label{eq: later oc eq3}
\E[ S_\nu(a) | (a_\nu \in A_o^\ell) \cap (a_{\nu-1} \in A_u)]\lesssim 2^{\ell}a2^{2\nu}.
\end{equation}

Now let $\nu > 1$. Note that the event $(a_{\nu-1} \in A_o^{\ell_1})$ implies that $(2^{\ell_1+1} - 1)a \le 2a_{\nu-1} - a < (2^{\ell_1 + 2} - 1)a$. By Remark \ref{rmk: occur}, the event $(a_\nu \in A_{o,\nu}^{\ell_\#})$ implies the occurrence of the event $E_\nu$ (recall that $E_\nu$ is the event that Epoch $\nu$ occurs). Therefore for $\ell_1 \ge \ell_\#$ we have
\[
        \prob( (a_\nu \in A_{o,\nu}^{\ell_\#})\cap(a_{\nu-1} \in A_o^{\ell_1})) \le \prob( E_\nu | (a_{\nu-1}\in A_o^{\ell_1})).
\]
We can then use Lemma \ref{lem: epoch nu large neg a} to get that
\[
\prob( E_\nu | (a_{\nu-1}\in A_o^{\ell_1})) \lesssim 2^{\ell_1} a \exp(-c 2^{2\nu}2^{\ell_1}a).
\]
Combining the last two inequalities and applying Remark \ref{rmk: e approx} gives
\begin{equation}\label{eq: later oc eq2.1}
\prob((a_\nu \in A_{o,\nu}^{\ell_\#}) \cap (a_{\nu-1} \in A_o^{\ell_1})) \lesssim  \frac{1}{2^{6\nu}2^{2\ell_1}a^2}.
\end{equation}
Similarly, note that the event $(a_{\nu-1} \in A_p)$ implies that $a \le 2a_{\nu-1} - a\le (2^{\ell_\#+1} -1)a$. We note that
\[
    \prob((a_\nu \in A_{o,\nu}^{\ell_\#})\cap(a_{\nu-1}\in A_p)) \le \prob(E_\nu | (a_{\nu-1}\in A_p))
\]
and apply Lemma \ref{lem: epoch nu large neg a} to get that
\[
\prob(E_\nu | (a_{\nu-1}\in A_p)) \lesssim a \exp(-c 2^{2\nu}a).
\]
Combining the last two inequalities and applying Remark \ref{rmk: e approx} gives
\begin{equation}\label{eq: later oc eq2.2}
\prob((a_\nu \in A_{o,\nu}^{\ell_\#})\cap(a_{\nu-1}\in A_p)) \lesssim \frac{1}{2^{4\nu}a}.
\end{equation}
Next, note that when $\nu \ge 2$ we can apply \eqref{eq: uc later eq1} to get
\begin{equation}\label{eq: later oc eq6}
\begin{split}
\prob\left((a_\nu \in A_o^{\ell_\#}) \cap (a_{\nu-1} \in A_u)\right) &\le \prob(a_{\nu-1} \in A_u) \\ &\lesssim \exp\left(-c 2^{2\nu}a\right)\\
&\lesssim \frac{1}{2^{4\nu}a^2}.
\end{split}
\end{equation}
When $\nu = 1$, we apply Lemma \ref{lem: epnu pc prob} to get that
\begin{equation}\label{eq: later oc eq6.1}
    \prob( a_1 \in A_o^{\ell_\#}) \le 1 - \prob((a_1 \in A_p)\cup E_1^c) \lesssim \exp(-ca)\lesssim \frac{1}{a^2}.
\end{equation}
Now suppose $\ell > \ell_\#$ and $\nu \ge 1$. Note that when $a_\nu \in A_o^\ell$ and $a_{\nu-1}\in A_u$ we have
\[
a_\nu - C_1 a_{\nu-1} \ge 2^\ell a - 2^{\ell_\#}a \ge 2^{\ell -1}a
\]
since $C_1 = 2^{\ell_\#}$. Moreover, $a_{\nu-1}\in A_u$ implies that $|a-2a_{\nu-1}| \le a$. We choose $\ell_\#$ to be sufficiently large to guarantee that
\[
\frac{1}{2^{\ell-1}} \le \frac{1}{2^{\ell_\#}} < \frac{1}{10 C_0 \log(2)};
\]
this allows us to apply Lemma \ref{lem: beta gamma} with $\gamma(x) := a$ and $\beta(x) := 2^{\ell-1}a$ for all $x \in A_u$ to get
\begin{equation}\label{eq: later oc eq7}
    \begin{split}
        \prob\left( (a_\nu \in A_o^\ell) \cap\right. &\left. (a_{\nu-1}\in A_u)\right) \le \prob\left( (a_\nu \in A_o^\ell) | (a_{\nu-1}\in A_u)\right)\\
        &\le \prob\left( (a_\nu - C_1 a_{\nu-1}) \ge 2^{\ell-1}a | a_{\nu-1}\in A_u\right) \\
        &\lesssim \exp\left(-c 2^{2\nu} 2^{\ell-1} a \right)\lesssim \frac{1}{2^{4\nu}2^{2\ell}a^2}.
    \end{split}
\end{equation}
Now suppose $\ell_2 > \ell_\#$ and $\nu > 1$. Then $(a_\nu \in A_{o,\nu}^{\ell_2})$ implies that
\[
a_\nu - C_1 a_{\nu-1} \ge 2^{\ell_2 -1}a_{\nu-1}
\]
(where we've used that $C_1 = 2^{\ell_\#}$). If $a_{\nu-1} \ge a$, then
\[
|a-2a_{\nu-1}|\le 2a_{\nu-1}.
\]
Again, we choose $\ell_\#$ to be sufficiently large to guarantee that
\[
\frac{1}{2^{\ell_2 - 2}} \le \frac{1}{2^{\ell_\#-1}} < \frac{1}{10C_0 \log(2)};
\]
this allows us to apply Lemma \ref{lem: beta gamma} with $\gamma(x) = 2x$ and $\beta(x) = 2^{\ell_2-1}x$ for $x \in A_o^{\ell_1}$ to get
\begin{equation}\label{eq: later oc eq8}
    \begin{split}
        \prob\left( (a_\nu \in A_{o,\nu}^{\ell_2}) \cap\right. &\left. (a_{\nu-1}\in A_o^{\ell_1})\right) \le \prob\left( (a_\nu \in A_{o,\nu}^{\ell_2}) | (a_{\nu-1}\in A_o^{\ell_1})\right)\\
        &\le \prob\left( (a_\nu - C_1 a_{\nu-1}) \ge 2^{\ell_2-1}a_{\nu-1} | (a_{\nu-1}\in A_o^{\ell_1})\right) \\
        &\lesssim \exp\left(-c 2^{2\nu} 2^{\ell_2+\ell_1}a \right) \lesssim \frac{1}{2^{4\nu}2^{2(\ell_1 + \ell_2)} a^2}
    \end{split}
\end{equation}
and again with $x \in A_p$ to get
\begin{equation}\label{eq: later oc eq9}
    \begin{split}
        \prob\left( (a_\nu \in A_{o,\nu}^{\ell_2}) \cap \right. & (a_{\nu-1}\left.\in A_p)\right) \le \prob\left( (a_\nu \in A_{o,\nu}^{\ell_2}) | (a_{\nu-1}\in A_p)\right)\\
        &\le \prob\left( (a_\nu - C_1 a_{\nu-1}) \ge 2^{\ell_2-1}a_{\nu-1} | (a_{\nu-1}\in A_p)\right) \\
        &\lesssim \exp\left(-c 2^{2\nu} 2^{\ell_2}a \right) \lesssim \frac{1}{2^{4\nu}2^{2\ell_2}a^2}.
    \end{split}
\end{equation}
Combining \eqref{eq: later oc eq1}, \eqref{eq: later oc eq2.1}, and \eqref{eq: later oc eq8} gives
\begin{equation}
\begin{split}
\sum_{\substack{\ell_1 \ge \ell_\# \\ \ell_2 \ge \ell_\#}} \E [ S_\nu(a) &| (a_\nu  \in A_{o,\nu}^{\ell_2}) \cap (a_{\nu-1}\in A_o^{\ell_1}) ] \cdot \prob\left( (a_\nu \in A_{o,\nu}^{\ell_2}) \cap (a_{\nu-1}\in A_o^{\ell_1})\right)\\
&\lesssim \sum_{\substack{\ell_1 \ge \ell_\# \\ \ell_2 \ge \ell_\#}} 2^{\ell_1 + \ell_2} a 2^{2\nu}\cdot\left( \frac{1}{2^{4\nu}2^{2(\ell_1+\ell_2)}a^2} \right)\\
&\lesssim 2^{-2\nu}
\end{split}
\end{equation}
when $\nu > 1$. Next, combine \eqref{eq: later oc eq2}, \eqref{eq: later oc eq2.2}, and \eqref{eq: later oc eq9} to get
\begin{equation}
    \begin{split}
        \sum_{\ell \ge \ell_\#} \E[S_\nu(a) &| (a_\nu  \in A_{o,\nu}^\ell)\cap (a_{\nu-1}\in A_p)] \cdot \prob \left( (a_\nu \in A_{o,\nu}^{\ell})\cap (a_{\nu-1}\in A_p)\right)\\
        &\lesssim \sum_{\ell \ge \ell_\#} 2^\ell a 2^{2\nu} \cdot\left(\frac{1}{2^{4\nu} 2^{2\ell} a}\right)\\
        & \lesssim 2^{-2\nu}
    \end{split}
\end{equation}
when $\nu > 1$. Finally, combine \eqref{eq: later oc eq3}, \eqref{eq: later oc eq6}, \eqref{eq: later oc eq6.1}, and \eqref{eq: later oc eq7} to get
\begin{equation}\label{eq: later oc eq10}
    \begin{split}
    \sum_{\ell \ge \ell_\#} \E[ S_\nu(a) &| (a_\nu \in A_o^\ell) \cap(a_{\nu-1}\in A_u)]\cdot \prob((a_\nu \in A_o^\ell) \cap (a_{\nu-1}\in A_u))\\
    &\lesssim \sum_{\ell \ge \ell_\#} 2^\ell a 2^{2\nu} \cdot \left(\frac{1}{2^{4\nu}2^{2\ell}a^2} \right) \\
    & \le 2^{-2\nu}
    \end{split}
\end{equation}
when $\nu \ge 1$. Combining these three bounds with \eqref{eq: ep nu oc score decomp} proves the lemma for $\nu >1$. Combining \eqref{eq: later oc eq10} and \eqref{eq: ep 1 oc score decomp} proves the lemma for $\nu = 1$.
\end{proof}

\section{When $|a|$ is bounded}\label{sec: a small}
In this section we prove Lemma \ref{lem: main 2}.\ref{lem: main 2 ii}, i.e. we show that there exists a constant $C>1$ depending only on $T$ such that $S_*(a) \le C$ for all $|a| \le 1$. We first remark that
\begin{equation}\label{eq: a bdd p0}
\sum_{\nu=-1}^0 \E[S_\nu(a)] = O(1)
\end{equation}
since $a_{-1} = a_0 = 0$ (and thus $u_{-1}=u_0=0$) and $|q(t)|\le 2$ for all $t \in [0,t_1]$. Next, recall that
\begin{equation}\label{eq: a_nu def}
a_\nu = C_0 \frac{\log(2)}{t_\nu-t_{\nu-1}}+C_1\cdot a_{\nu-1},
\end{equation}
where $C_0 = 4$ and $C_1 = 2^{\ell_\#}$. Since we always have $t_\nu - t_{\nu-1} \le T$, whenever we reach Epoch $\nu$ we have $a_\nu \ge \frac{C_0 \log(2) C_1^{\nu-1}}{T} > \frac{2 C_1^{\nu-1}}{T}$. We define
\[
\nu_* := \left\lceil \frac{\log(T\cdot 2^{\ell_\#})}{\log(C_1)} \right\rceil + 1;
\]
observe that $\nu_* = O(1)$. Note that if we reach Epoch $\nu_*$ we are guaranteed to have
\[
a_{\nu_*} > 2^{\ell_\#+1}.
\]
This implies that when $\nu \ge \nu_*+1$ we have
\begin{equation}\label{eq: a bdd 1v}
\E[S_\nu(a)] = \E[S_\nu(a)|(a_{\nu-1}\ge 1)]\cdot \prob(a_{\nu-1}\ge 1).
\end{equation}
We will show that when $\nu \ge 2$ we have
\begin{equation}\label{eq: a bdd 2v}
\E[S_\nu(a)|(a_{\nu-1}\ge 1)]\cdot \prob(a_{\nu-1}\ge 1) = O\left(\frac{1}{2^{2\nu}}\right).
\end{equation}
Combining \eqref{eq: a bdd 1v} and \eqref{eq: a bdd 2v} gives
\begin{equation}\label{eq: a bdd 3up}
\sum_{\nu=\nu_*+1}^\infty\E[S_\nu(a)] \le \sum_{\nu=\nu_*+1}^\infty \frac{C}{2^{2\nu}} = O(1).
\end{equation}
We now prove \eqref{eq: a bdd 2v}. Define
\[
\widehat{A}_o^{\ell} := \{ b \in \R : 2^{\ell} \le b < 2^{\ell+1}\}
\]
for $\ell\ge 0$. Fix $\nu \ge 2$. We have
\begin{equation}\label{eq: a med main 1}
\begin{split}
\E[S_\nu(a) | (a_{\nu-1} \ge 1)]&\cdot \prob(a_{\nu-1}\ge 1)\\
&= \sum_{\ell_1 = 0}^\infty \E[S_\nu(a) | (a_{\nu-1}\in \widehat{A}_o^{\ell_1})] \cdot \prob(a_{\nu-1}\in \widehat{A}_o^{\ell_1}).
\end{split}
\end{equation}
Given $a_{\nu-1}>0$, we define
\[
A_{o,\nu}^{\ell} := \{ b \in \R : 2^{\ell}a_{\nu-1} \le b < 2^{\ell+1}a_{\nu-1}\}
\]
for $\ell \ge 0$. Note that $a_\nu \ge 2^{\ell_\#} a_{\nu-1}$; this follows from \eqref{eq: a_nu def}.
Continuing from \eqref{eq: a med main 1}, we get
\begin{equation}\label{eq: a med main 2}
\begin{split}
\E&[S_\nu(a) | (a_{\nu-1} \ge 1)]\cdot \prob(a_{\nu-1}\ge 1) \\
&= \sum_{\substack{\ell_1 \ge 0\\ \ell_2 \ge \ell_\#}} \E[S_\nu(a) | (a_\nu \in A_{o,\nu}^{\ell_2}) \cap(a_{\nu-1}\in \widehat{A}_o^{\ell_1})] \cdot \prob((a_\nu \in A_{o,\nu}^{\ell_2}) \cap (a_{\nu-1}\in \widehat{A}_o^{\ell_1})).
\end{split}
\end{equation}
The event $(a_\nu \in A_{o,\nu}^{\ell_2}) \cap(a_{\nu-1} \in \widehat{A}_o^{\ell_1})$, along with the assumption that $|a|\le 1$, implies that $a_\nu \le 2^{\ell_1+\ell_2+2}$ and $2a_\nu - a \ge 2^{\ell_1+\ell_2+1}-1 \ge 2^{\ell_1+\ell_2}$. Therefore we can apply Lemma \ref{lem:gen score f} to get
\begin{equation}
\E[S_\nu(a) | (a_\nu \in A_{o,\nu}^{\ell_2}) \cap(a_{\nu-1}\in \widehat{A}_o^{\ell_1})] \lesssim 2^{2\nu}2^{\ell_2+\ell_1}.
\end{equation}
Note that
\begin{equation*}
\begin{split}
\prob((a_\nu \in A_{o,\nu}^{\ell_\#}) \cap (a_{\nu-1}\in \widehat{A}_o^{\ell_1})) &\le \prob((a_\nu \in A_{o,\nu}^{\ell_\#}) | (a_{\nu-1}\in \widehat{A}_o^{\ell_1}))\\
&\le \prob(E_\nu | (a_{\nu-1} \in \widehat{A}_o^{\ell_1})).
\end{split}
\end{equation*}
Since the event $(a_{\nu-1} \in \widehat{A}_o^{\ell_1})$ for $\ell_1 \ge 0$ implies that $2a_{\nu-1} - a \le 2^{\ell_1+2} + 1\lesssim 2^{\ell_1}$ and that $2a_{\nu-1} -a \ge 2^{\ell_1+1} - 1 \ge 2^{\ell_1}$, we can apply Lemma \ref{lem: epoch nu large neg a} to get
\[
\prob(E_\nu | (a_{\nu-1} \in \widehat{A}_o^{\ell_1})) \lesssim 2^{\ell_1} \exp(-c 2^{2\nu}2^{\ell_1}).
\]
Therefore
\begin{equation}
    \prob((a_\nu \in A_{o,\nu}^{\ell_\#}) \cap (a_{\nu-1}\in \widehat{A}_o^{\ell_1})) \lesssim 2^{\ell_1} \exp(-c 2^{2\nu}2^{\ell_1}).
\end{equation}
Now suppose that $\ell_2 > \ell_\#$. In this case, the event $(a_\nu \in A_{o,\nu}^{\ell_2})$ implies that
\[
a_\nu - C_1 a_{\nu-1} \ge (2^{\ell_2} - 2^{\ell_\#})a_{\nu-1}\ge 2^{\ell_2 - 1} a_{\nu-1}.
\]
Therefore,
\begin{equation*}
\begin{split}
\prob((a_\nu \in A_{o,\nu}^{\ell_2})\cap & (a_{\nu-1}\in \widehat{A}_o^{\ell_1}))\\
&\le \prob(a_\nu \in A_{o,\nu}^{\ell_2}|a_{\nu-1}\in \widehat{A}_o^{\ell_1})\\
&\le \prob(a_\nu - C_1 a_{\nu-1} \ge 2^{\ell_2-1}a_{\nu-1}|a_{\nu-1}\in \widehat{A}_o^{\ell_1}).
\end{split}
\end{equation*}
Note that $|a-2a_{\nu-1}| \le 3 a_{\nu-1}$ (since we assume $|a|\le 1 \le a_{\nu-1}$). We are going to apply Lemma \ref{lem: beta gamma} with $\beta(x) = 2^{\ell_2-1}x$, $\gamma(x) = 3x$, and $X = \widehat{A}_o^{\ell_1}$. Note that in this case $\beta^* = 2^{\ell_2 + \ell_1 -1}$. We choose $\ell_\#$ to be large enough to ensure that the hypothesis
\[
\frac{\gamma(x)}{\beta(x)} = \frac{3}{2^{\ell_2-1}} \le \frac{3}{2^{\ell_\#}} < \frac{1}{C_0 10 \log(2)}.
\]
holds. Applying Lemma \ref{lem: beta gamma} gives
\[
\prob(a_\nu - C_1 a_{\nu-1} \ge 2^{\ell_2-1}a_{\nu-1}|a_{\nu-1}\in \widehat{A}_o^{\ell_1}) \lesssim \exp(-c 2^{2\nu}2^{\ell_2+\ell_1}).
\]
Consequently,
\begin{equation}\label{eq: med 3}
   \prob((a_\nu \in A_{o,\nu}^{\ell_2})\cap  (a_{\nu-1}\in \widehat{A}_o^{\ell_1}))\lesssim \exp(-c 2^{2\nu}2^{\ell_2+\ell_1}). 
\end{equation}
Combining \eqref{eq: a med main 2} - \eqref{eq: med 3} and applying Remark \ref{rmk: e approx}, we get that
\begin{equation*}
\begin{split}
\E[&S_\nu(a) | (a_{\nu-1} \ge 1)]\cdot \prob(a_{\nu-1}\ge 1) \\
&\le \sum_{\ell_1 = 0}^\infty 2^{2\nu}2^{2\ell_1}\exp(-c 2^{2\nu}2^{\ell_1})
+ \sum_{\substack{\ell_1 \ge 0 \\ \ell_2 > \ell_\#}} 2^{2\nu}2^{\ell_2 + \ell_1} \exp(-c 2^{2\nu}2^{\ell_2+\ell_1})\\
& =O\left( \frac{1}{2^{2\nu}}\right).
\end{split}
\end{equation*}
This proves \eqref{eq: a bdd 2v}. By \eqref{eq: a bdd p0} and \eqref{eq: a bdd 3up}, we will have proved Lemma \ref{lem: main 2}.\ref{lem: main 2 ii} once we show that
\begin{equation}\label{eq: a bdd 12}
\sum_{\nu=1}^{\nu_*}\E[S_\nu(a)] = O(1).
\end{equation}
Note that \eqref{eq: a bdd 2v}, which we've just proved, implies that
\[
\E[S_{\nu}(a)] = \E[S_{\nu}(a) | a_{\nu-1} < 1] \cdot \prob(a_{\nu-1} <1) + O(1)
\]
for $2 \le \nu \le \nu_*$. Therefore, since $a_0 = 0$ and $\nu_* = O(1)$, \eqref{eq: a bdd 12} is implied by showing that
\begin{equation}\label{eq: a bdd 3v}
    \E[S_\nu(a) | a_{\nu-1} < 1] \cdot \prob( a_{\nu-1}<1) = O(1)
\end{equation}
for any $1 \le \nu \le \nu_*$. We first claim that
\[
\E[S_\nu(a) | (a_\nu \le 2^{\ell_\#+1})\cap(a_{\nu-1}<1)] = O(1)
\]
when $1 \le \nu \le \nu_*$. Recall that we have $|q(t)|\le 2^{\nu_*+1}$ when $t \in [t_1,t_{\nu_*+1}]$ and $u(t) = -2 a_\nu q(t)$ when $t \in [t_\nu, t_{\nu+1}]$. Therefore the assumption $(a_\nu \le 2^{\ell_\#+1})$ implies that $|u(t)| =O(1)$ for all $t \in [t_1, t_{\nu_*+1}]$. This shows that when $1 \le \nu \le \nu_*$, we have
\begin{equation}
\begin{split}
\E&[S_\nu(a) | (a_\nu \le 2^{\ell_\#+1})\cap(a_{\nu-1}<1)]
\\&=
\E\left[\left.\int_{t_\nu}^{t_{\nu+1}} (q^2(t) + u^2(t)) dt \right | (a_\nu \le 2^{\ell_\#+1})\cap(a_{\nu-1}<1)\right] = O(1).
\end{split}
\end{equation}
Thus, in order to prove \eqref{eq: a bdd 3v}, it remains to control
\[
\E[S_\nu(a)|(a_\nu > 2^{\ell_\#+1})\cap (a_{\nu-1}<1)]\cdot \prob((a_\nu > 2^{\ell_\#+1})\cap(a_{\nu-1}<1)).
\]
for $1 \le \nu \le \nu_*$. Note that we have
\begin{equation}\label{eq: med 8}
\begin{split}
\E&[ S_{\nu}(a) | (a_\nu > 2^{\ell_\#+1})\cap (a_{\nu-1} < 1) ] \cdot \prob((a_\nu > 2^{\ell_\#+1})\cap (a_{\nu-1} < 1))\\
&= \sum_{\ell=\ell^\#+1}^\infty \E[S_{\nu}(a) | (a_{\nu} \in \widehat{A}_o^\ell) \cap (a_{\nu-1} < 1)] \cdot \prob( (a_{\nu-1}< 1) \cap (a_{\nu} \in \widehat{A}_o^{\ell})).
\end{split}
\end{equation}
The event $(a_{\nu}\in \widehat{A}_o^\ell)\cap(a_{\nu-1}<1)$ (and the assumption that $|a|\le 1$) implies that $a_{\nu} \le 2^{\ell+1}$ and $2a_{\nu}-a\ge 2^{\ell+1}-1\ge 2^\ell$. Therefore we can apply Lemma \ref{lem:gen score f} to get
\begin{equation}\label{eq: med 9}
    \E[S_{\nu}(a) | (a_{\nu} \in \widehat{A}_o^{\ell}) \cap (a_{\nu-1} < 1)]\lesssim 2^\ell.
\end{equation}
Next, observe that the event $(a_{\nu}\in \widehat{A}_o^\ell)$ given the event $(a_{\nu-1}<1)$, along with the fact that $\ell > \ell_\#$, implies that
\[
a_{\nu} - C_1a_{\nu-1} > 2^\ell - 2^{\ell_\#} \ge 2^{\ell-1}.
\]
Therefore
\begin{equation}
    \begin{split}
    \prob( (a_{\nu-1}< 1) \cap (a_{\nu} \in \widehat{A}_o^{\ell}))&\le \prob( (a_{\nu} \in \widehat{A}_o^{\ell})|(a_{\nu-1}< 1))\\
    &\le\prob(a_{\nu} - C_1 a_{\nu-1} > 2^{\ell-1}|(a_{\nu-1} < 1)).
    \end{split}
\end{equation}
Note that $|2a_{\nu-1}-a| \le 3$. We will apply Lemma \ref{lem: beta gamma} with $X = \{ b \in \R : b < 1\}$, $\beta(x) = 2^{\ell-1}$, and $\gamma(x) = 3 $. We choose $\ell_\#$ to be sufficiently large to guarantee that
\[
\frac{\gamma(x)}{\beta(x)} = \frac{3 }{2^{\ell-1}} < \frac{3}{2^{\ell_\#}} < \frac{1}{10 C_0 \log(2)}.
\]
Then we apply Lemma \ref{lem: beta gamma} to get that
\begin{equation}\label{eq: med 10}
    \prob(a_\nu - C_1 a_{\nu-1} > 2^{\ell-1} | (a_{\nu-1}<1))\lesssim \exp(-c 2^{\ell-1}).
\end{equation}
Combining \eqref{eq: med 8} - \eqref{eq: med 10} gives
\begin{equation}\label{eq: med 11}
    \begin{split}
        \E[ S_{\nu}(a) &| (a_\nu > 2^{\ell_\#+1})\cap (a_{\nu-1} < 1) ] \cdot \prob((a_\nu > 2^{\ell_\#+1})\cap(a_{\nu-1} < 1)) \\
        &\lesssim \sum_{\ell =\ell^\#+1}^\infty 2^\ell   \exp(-c 2^{\ell-1}) \\
        &= O(1).
    \end{split}
\end{equation}
This completes the proof of \eqref{eq: a bdd 3v}, and thus completes the proof of Lemma \ref{lem: main 2}.\ref{lem: main 2 ii}.

\section{When $a$ is large and negative}\label{sec: a neg}
In this section we prove Lemma \ref{lem: main 2}.\ref{lem: main 2 iii}, i.e. we show that for all $a \le -1$ we have $S_*(a) \lesssim \frac{1}{|a|}$. We first note that
\[
S_*(a) = \sum_{\nu=-1}^\infty \E[S_\nu(a)].
\]
Since $a_{-1}=a_0 = 0$, we apply Lemma \ref{lem:gen score f} with $m = |a|$ and $M=0$ to get
\[
\E[S_{-1}(a) + S_0(a)] \lesssim \frac{1}{|a|}.
\]
It remains to show that
\[
\sum_{\nu=1}^\infty \E[S_\nu(a)] = O\left(\frac{1}{|a|}\right).
\]
Let $A_{o,\nu}^\ell$ be as in Section \ref{sec: oc}. Define
\[
B_{o}^\ell := \{ b \in \R : 2^\ell |a| \le b < 2^{\ell+1} |a|\}.
\]
When $\nu\ge 2$, we have
\begin{equation}\label{eq: a large neg main decomp}
\begin{split}
\E[S_\nu(a)] &\le \sum_{\ell_1 \ge 0} \E[S_\nu(a) | (a_{\nu-1} \in B_o^{\ell_1})]\cdot \prob(a_{\nu-1}\in B_o^{\ell_1}) \\
&+ \E[S_\nu(a)|(a_{\nu-1} < |a|) ] \cdot \prob( a_{\nu-1}<|a|).
\end{split}
\end{equation}
When $\nu = 1$, we have
\begin{equation}\label{eq: a large neg nu 1}
    \E[S_1(a)] =\E[S_\nu(a)|(a_{\nu-1} < |a|) ] \cdot \prob( a_{\nu-1}<|a|).
\end{equation}
We first show how to control the first term in \eqref{eq: a large neg main decomp}. Recall (see \eqref{eq: anu def}) that $a_\nu \ge 2^{\ell_\#} a_{\nu-1}$. Therefore
\begin{equation}\label{eq: a neg 0}
\begin{split}
\E[S_\nu(a) | & (a_{\nu-1} \in B_o^{\ell_1})]\cdot \prob(a_{\nu-1}\in B_o^{\ell_1}) \\
&= \sum_{\ell_2 \ge \ell_\#}\E[S_\nu(a) |(a_\nu \in A_{o,\nu}^{\ell_2})\cap(a_{\nu-1}\in B_o^{\ell_1})]\cdot \prob\left((a_\nu \in A_{o,\nu}^{\ell_2})\cap(a_{\nu-1}\in B_o^{\ell_1})\right).
\end{split}
\end{equation}
The event $(a_\nu \in A_{o,\nu}^{\ell_2})\cap(a_{\nu-1}\in B_o^{\ell_1})$ implies that $a_\nu \le 2^{\ell_2+\ell_1+2}|a|$ and $2a_\nu - a \ge 2^{\ell_2+\ell_1}|a|$. Therefore we can apply Lemma \ref{lem:gen score f} to get
\begin{equation}\label{eq: a neg 1}
\E[S_\nu(a) | (a_\nu \in A_{o,\nu}^{\ell_2}) \cap (a_{\nu-1}\in B_o^{\ell_1})]\lesssim 2^{\ell_1+\ell_2}\cdot 2^{2\nu}\cdot |a|.
\end{equation}
Suppose $\ell_2 > \ell_\#$. Then the event $(a_\nu \in A_{o,\nu}^{\ell_2})\cap(a_{\nu-1}\in B_o^{\ell_1})$ implies that
\[
a_\nu - C_1 a_{\nu-1} \ge (2^{\ell_2}-2^{\ell_\#})\cdot a_{\nu-1} \ge 2^{\ell_2-1 + \ell_1} |a|
\]
and that
\[
|a-2a_{\nu-1}| \le |a| + 2^{\ell_1 + 2}|a| \le 2^{\ell_1+3}|a|.
\]
We are going to apply Lemma \ref{lem: beta gamma} with $\gamma(x) = 2^{\ell_1+3}|a|$ and $\beta(x) = 2^{\ell_1+\ell_2-1}|a|$ and $X = B_o^{\ell_1}$. Choosing $\ell_\#$ sufficiently large guarantees that
\[
\frac{\gamma(x)}{\beta(x)} = \frac{1}{2^{\ell_2 - 4}} < \frac{1}{2^{\ell_\# - 4}} < \frac{1}{10 C_0 \log(2)}.
\]
Therefore we apply Lemma \ref{lem: beta gamma} to get that
\begin{equation}\label{eq: a neg 2}
    \begin{split}
\prob((a_\nu \in A_{o,\nu}^{\ell_2})\cap (&a_{\nu-1}\in B_o^{\ell_1})) \le \prob((a_\nu \in A_{o,\nu}^{\ell_2})| (a_{\nu-1} \in B_o^{\ell_1}))\\
&\le \prob( a_\nu - C_1 a_{\nu-1} \ge |a|\cdot 2^{\ell_2 + \ell_1 - 1} | (a_{\nu-1}\in B_o^{\ell_1}))\\
&\lesssim \exp(-c 2^{2\nu}2^{\ell_1+\ell_2}|a|)\\
&\lesssim \frac{1}{2^{4\nu}2^{2(\ell_1+\ell_2)}|a|^2}.
    \end{split}
\end{equation}
Next, note that
\[
\prob((a_\nu \in A_{o,\nu}^{\ell_\#})\cap (a_{\nu-1}\in B_o^{\ell_1})) \le \prob(E_\nu | (a_{\nu-1}\in B_o^{\ell_1})).
\]
Since the event $(a_{\nu-1}\in B_o^{\ell_1})$ implies that $2a_{\nu-1}-a \le 2^{\ell_1+2}|a| + |a| \le 2^{\ell_1+3}|a|$ and $2a_{\nu-1}-a\ge 2^{\ell_1+1}|a| + |a| \ge 2^{\ell_1+1}|a|$, we can apply Lemma \ref{lem: epoch nu large neg a} to get
\[
\prob(E_\nu | (a_{\nu-1}\in B_o^{\ell_1})) \lesssim 2^{\ell_1}|a|\exp(-c2^{2\nu}2^{\ell_1}|a|).
\]
The last two equations and Remark \ref{rmk: e approx} give
\begin{equation}\label{eq: a neg 3}
    \prob((a_\nu \in A_{o,\nu}^{\ell_\#})\cap (a_{\nu-1}\in B_o^{\ell_1}))\lesssim \frac{1}{2^{6\nu}2^{2\ell_1}|a|^2}.
\end{equation}
Combining \eqref{eq: a neg 0} - \eqref{eq: a neg 3} gives
\begin{equation}\label{eq: a neg *}
\begin{split}
    \sum_{\ell_1 \ge 0}\E[S_\nu(a)|&(a_{\nu-1}\in B_o^{\ell_1})]\cdot \prob(a_{\nu-1}\in B_o^{\ell_1})\\ &\lesssim \sum_{\substack {\ell_2 \ge \ell_\# \\\ell_1 \ge 0}} 2^{\ell_1+\ell_2}2^{2\nu}|a|\left(\frac{1}{2^{4\nu}2^{2(\ell_1+\ell_2)}|a|^2}\right)\lesssim \frac{1}{2^{2\nu}|a|}.
\end{split}
\end{equation}
Next, we claim that
\begin{equation}\label{eq: a neg ***}
    \E[S_\nu(a) | (a_{\nu-1}< |a|)] \cdot \prob( a_{\nu-1}< |a|) = O\left( \frac{1}{2^{2\nu}|a|}\right)
\end{equation}
for $\nu \ge 1$. Combining this with \eqref{eq: a large neg main decomp}, \eqref{eq: a large neg nu 1}, and \eqref{eq: a neg *} gives
\[
\sum_{\nu=1}^\infty \E[S_\nu(a)] \lesssim \sum_{\nu=1}^\infty \frac{1}{2^{2\nu}|a|} \lesssim \frac{1}{|a|}.
\]
This completes the proof of Lemma \ref{lem: main 2}.\ref{lem: main 2 iii}, and thus it just remains to establish \eqref{eq: a neg ***}.

We have
\begin{equation}\label{eq: a neg 4}
    \begin{split}
\E[S_\nu(a) | &(a_{\nu-1}< |a|)] \cdot \prob( a_{\nu-1}< |a|)\\
&= \sum_{\ell_1 \ge \ell_\#} \E[S_\nu(a)|(a_\nu \in A_o^{\ell_1})\cap(a_{\nu-1}<|a|)]\cdot\prob((a_\nu\in A_o^{\ell_1})\cap(a_{\nu-1}<|a|)).
    \end{split}
\end{equation}
Since the event $(a_\nu \in A_o^{\ell_1})\cap(a_{\nu-1}<|a|)$ implies that $a_\nu \le 2^{\ell_1+1}|a|$ and $2a_{\nu}-a\ge 2^{\ell_1+1}|a|+|a|\ge 2^{\ell_1+1}|a|$, we apply Lemma \ref{lem:gen score f} to get
\begin{equation}\label{eq: a neg 5}
\E[S_\nu(a) | (a_\nu \in A_o^{\ell_1})\cap(a_{\nu-1} < |a|)] \lesssim 2^{2\nu}2^{\ell_1}|a|.
\end{equation}
For $\ell_1 > \ell_\#$, the event $(a_\nu \in A_o^{\ell_1})\cap(a_{\nu-1} < |a|)$ implies that
\[
a_\nu - C_1 a_{\nu-1} \ge (2^{\ell_1} - 2^{\ell_\#})\cdot |a| \ge 2^{\ell_1-1}|a|
\]
and that
\[
|a - 2a_{\nu-1}| \le 3 \cdot |a|.
\]
We will apply Lemma \ref{lem: beta gamma} with $\beta(x) = 2^{\ell_1-1}\cdot |a|$ and $\gamma(x) = 3\cdot|a|$. Choosing $\ell_\#$ sufficiently large guarantees that
\[
\frac{\gamma(x)}{\beta(x)} = \frac{3}{2^{\ell_1-1}} < \frac{3}{2^{\ell_\#-1}}< \frac{1}{10 C_0 \log(2)}.
\]
Therefore we can apply Lemma \ref{lem: beta gamma} to get
\begin{equation}\label{eq: a neg 6}
\begin{split}
\prob((a_\nu \in A_o^{\ell_1}) \cap (a_{\nu-1}<|a|)) &\le \prob(a_\nu \in A_o^{\ell_1}|a_{\nu-1}<|a|) \\
&\le \prob(a_\nu - C_1 a_{\nu-1} \ge 2^{\ell_1-1}|a| | (a_{\nu-1}<|a|)) \\
&\lesssim \exp(-c 2^{2\nu} 2^{\ell_1-1} |a|)\\
&\lesssim \frac{1}{2^{4\nu}2^{2\ell_1}|a|^2}.
\end{split}
\end{equation}
Since the event $(a_{\nu-1}<|a|)$ implies that $2a_{\nu-1}-a \le 3|a|$ and $2a_{\nu-1} - a \ge |a|$, we can apply Lemma \ref{lem: epoch nu large neg a} to get
\begin{equation}\label{eq: a neg 7}
\begin{split}
\prob((a_\nu \in A_o^{\ell_\#})\cap(a_{\nu-1}<|a|))&\le \prob(E_\nu | (a_{\nu-1}<|a|))\\
&\lesssim |a| \exp(-c2^{2\nu}|a|)\\
&\lesssim \frac{1}{2^{6\nu}|a|^2}.
\end{split}
\end{equation}
Combining \eqref{eq: a neg 4}-\eqref{eq: a neg 7} gives
\begin{equation}\label{eq: a neg **}
\E[S_\nu(a) | (a_{\nu-1}< |a|)] \cdot \prob( a_{\nu-1}< |a|) \lesssim \sum_{\ell_1\ge \ell_\#}\frac{1}{2^{2\nu}2^{\ell_1}|a|}=O\left(\frac{1}{2^{2\nu}|a|}\right).
\end{equation}
This proves \eqref{eq: a neg ***}, finishing the proof of the lemma.

\section{The optimal strategy for known $a$}\label{sec: S_0}
The goal of this section is to prove Lemma \ref{lem: S_0}. We define the \emph{expected cost-to-go} of the optimal strategy $\sigma_\opt(a)$ at time $t$ and position $q$ by
\begin{equation}\label{eq: s0}
J_0(q,t;a) = \E\left[ \int_t^T (q^2 + u^2) dt\right] \; .
\end{equation}
Note that \eqref{eq: s0} is more general than the quantity $S_\opt(a)$ introduced in Section \ref{sec:Main_res}, but that we have $S_\opt(a) = J_0(0,0;a)$. We begin by deriving a Hamilton-Jacobi-Bellman equation for $J_0$. For a small time increment $\Delta t$, we have
\begin{align}
    J_0(q,t;a) = (q^2 + u^2)\Delta t + \E[J_0(q+\Delta q,t + \Delta t;a)],
\end{align}
where $\Delta q$ and $\Delta t$ are the corresponding increments of $q$ and $t$. Expanding the last term in a Taylor series, we obtain
\begin{align*}
J_0(q,t;a) = (q^2+u^2)\Delta t& + J_0(q,t;a) + (\Delta t) \partial_t J_0(q,t;a) + \E[\Delta q]\partial_q J_0(q,t;a) \\&+ \frac{1}{2}\E[(\Delta q)^2]\partial_q^2 J_0(q,t;a) + o(\Delta t).
\end{align*}
Equation \eqref{eq: q ode} implies that
\begin{align}
    &\E[\Delta q] = (q a + u)\Delta t \\
    &\E[\Delta q^2] = \Delta t.
\end{align}
Thus, after dividing by $\Delta t$ and taking $\Delta t \rightarrow 0$, we obtain for an optimal strategy $u$
\begin{align}
0 = \min_{u \in \R} \left\{ \partial_t J_0 + (qa + u)\partial_q J_0 + \frac{1}{2}\partial_q^2 J_0 + q^2 + u^2\right\}
\end{align}
with $J_0(q,T;a) = 0$. The minimum of the RHS with respect to $u$ occurs when 
\begin{align}
    u = -\frac{1}{2}\partial_q J_0.
\end{align}
We thus arrive at a final PDE
\begin{align}
0 = \partial_t J_0 + (qa)\partial_q J_0 + \frac{1}{2}\partial_q^2 J_0 + q^2 - \frac{1}{4}\left(\partial_q J_0\right)^2.
\end{align}
We can now guess a solution to $J_0$ of the form
\begin{equation}\label{eq: S_0}
J_0(q,t;a) = p(t;a)q^2 + r(t;a) \; ,
\end{equation}
where $p(t;a)$ and $r(t;a)$ are solutions to the following differential equations:
\begin{align}
-p'(t;a) &= 2ap(t;a) + 1 - p^2(t;a), &p(T;a) =0 \label{eq: p} \\
-r'(t;a) &= p, &r(T;a)=0 \label{eq: r}
\end{align}
Note that
\[
S_\opt(a) = J_0(0,0;a) = r(0;a).
\]
Solving \eqref{eq: p} explicitly gives
\begin{equation}\label{eq: p sol}
p(t;a) = a - \sqrt{a^2 + 1} \tanh\left( (t-T)\sqrt{a^2 + 1} + \tanh^{-1}\left(\frac{a}{\sqrt{a^2+1}}\right)\right) \;.
\end{equation}
We integrate \eqref{eq: p sol} from $0$ to $T$ to get
\begin{align}
\begin{split}\label{eq: r sol}
r(0;a) = aT &- \log\left(\cosh\left(\tanh^{-1}\left(\frac{a}{\sqrt{a^2+1}}\right)\right)\right) \\
& + \log\left(\cosh\left(\tanh^{-1}\left(\frac{a}{\sqrt{a^2+1}}\right) - T\sqrt{a^2+1}\right)\right) \; .
\end{split}
\end{align}
Using the identity $\cosh\left(\tanh^{-1}(x)\right) = \frac{1}{\sqrt{1-x^2}}$, \eqref{eq: r sol} gives
\begin{align}
\begin{split}\label{eq: r simp}
r(0;a) = aT &- \log(\sqrt{a^2+1})\\
&+ \log\left(\cosh\left(\tanh^{-1}\left(\frac{a}{\sqrt{a^2+1}}\right) - T\sqrt{a^2+1}\right)\right) \; .
\end{split}
\end{align}
Next, we use the identity $\tanh^{-1}(x) = \frac{1}{2}\log\left(\frac{1+x}{1-x}\right)$ to get
\begin{equation}\label{eq: tanh log}
\tanh^{-1}\left(\frac{a}{\sqrt{a^2+1}}\right) = \frac{1}{2}\log\left( \frac{\sqrt{a^2+1} + a}{\sqrt{a^2+1}-a}\right) = \log\left(\sqrt{a^2+1}+a\right). \;
\end{equation}
Combining this with \eqref{eq: r simp} gives
\begin{equation}
\begin{split}\label{eq: r simp 2}
r(0;a) = aT &- \log(\sqrt{a^2+1})\\
&+ \log\left(\cosh\left(\log\left(\sqrt{a^2+1}+a\right) - T\sqrt{a^2+1}\right)\right) \; .
\end{split}
\end{equation}
One can check that
\[
S_\opt'(a) = \frac{\partial}{\partial a}\left( r(0;a)\right) > 0.
\]
This shows that $S_\opt$ is an increasing, continuous function of $a$. One can also check that $S_\opt(-1)=r(0;-1) >0$. Combining these facts proves Lemma \ref{lem: S_0}.\ref{lem: S_0 1}.

%% known large pos a
\subsection{When $a$ is large and positive}
We now prove Lemma \ref{lem: S_0}.\ref{lem: S_0 2}, i.e. we show that $S_\opt(a) \gtrsim a$ when $a\ge1$. Since $S_\opt$ is continuous, increasing, and nonnegative it suffices to show that $S_\opt(a) \gtrsim a$ for all $a$ sufficiently large.

For large enough $a$, there exists some $C>0$ so that
\[
T\sqrt{a^2+1} - \log\left(\sqrt{a^2+1}+1\right)\ge C a.
\]
Note that $\cosh(x) \ge \frac{e^x}{2}$ and $\cosh(x)$ is increasing for $x>0$; this gives
\[
\log\left(\cosh\left(T\sqrt{a^2+1} - \log\left(\sqrt{a^2+1}+1\right)\right)\right) \ge \log\left(\frac{e^{Ca}}{2}\right).
\]
Combining this with \eqref{eq: r simp 2} gives
\[
S_\opt(a) = r(0;a) \ge aT + \log\left(\frac{e^{Ca}}{2}\right)\ge C' a
\]
for sufficiently large $a$.

%% known large neg a
\subsection{When $a$ is large and negative}
We now prove Lemma \ref{lem: S_0}.\ref{lem: S_0 3}, i.e. we show that $S_\opt(a) \gtrsim \frac{1}{\vert a \vert}$ when $a \le -1$. As in the previous section, it suffices to prove that $S_\opt(a) \gtrsim \frac{1}{\vert a \vert}$ for $a<0$ with $|a|$ sufficiently large.

Note that $\log(\sqrt{a^2+1}+a) = - \log(\sqrt{a^2+1}-a)$. Combining this with \eqref{eq: r simp 2} gives
\[
S_\opt(a) = aT - \log\left(\sqrt{a^2+1}\right) + \log\left(\cosh\left(\log\left(\sqrt{a^2+1}+|a|\right)+T\sqrt{a^2+1}\right)\right).
\]
Note that
\begin{equation*}
    \begin{split}
       \log\left(\cosh\left(\log\left(\sqrt{a^2+1}+|a|\right)+T\sqrt{a^2+1}\right)\right)\\
       \ge T\sqrt{a^2+1} +\log\left(\frac{\sqrt{a^2+1}+|a|}{2}\right).
    \end{split}
\end{equation*}
Therefore there exists $c>0$ such that
\[
S_\opt(a) \ge  T\sqrt{a^2+1} - T|a| + \log\left(\frac{\sqrt{a^2+1}+|a|}{2\sqrt{a^2+1}}\right)\ge \frac{c}{|a|}
\]
for $a<0$ with $|a|$ sufficiently large.

%\noindent {\color{blue} \textbf{Conclusions + future work?}

\bibliographystyle{abbrv}
\bibliography{ref}

\end{document}